\documentclass[a4paper,8pt]{article}
\usepackage{geometry}
\usepackage{wrapfig}
\usepackage{lastpage}
\usepackage[T1]{fontenc}
\usepackage{geometry}
\usepackage[pdftex]{graphicx}
\usepackage{amsthm}
\usepackage{amssymb}
\usepackage{amsmath}
\usepackage[T1,T2A]{fontenc}
\usepackage[utf8]{inputenc}
\usepackage[russian,main=english]{babel}
\usepackage{setspace}
\usepackage{esint}
\usepackage{comment}
\usepackage{amsfonts}
\usepackage{fullpage}
\usepackage{parskip}
\usepackage{indentfirst}
\usepackage{caption}
\usepackage{subcaption}
\usepackage{listings}
\usepackage{xcolor}
\usepackage{hyperref}
\usepackage{float}
\usepackage{booktabs}
\usepackage{arcs}
\usepackage{cmap}
\usepackage{mathtools}

\usepackage{enumitem}
\newcommand{\subscript}[2]{$#1 _ #2$}
\hypersetup{colorlinks,
linkcolor = red, citecolor=blue}
\newcommand{\ov}{\overline}
\newcommand{\pr}{{\bf{P}}}
\newcommand{\bfx}{{\bf{x}}}

\newcommand{\ex}{{\bf{E}}}

\newcommand{\prf}{{\bf{P}_f}}
\newcommand{\exf}{{\bf{E}_f}}
\newcommand{\dist}{\mathrm{dist}}
\newcommand{\bol}{\boldsymbol}
\newcommand{\ab}{\indent}
\newcommand{\RomanNumeralCaps}[1]
{\MakeUppercase{\romannumeral #1}}
\addto\captionsrussian{}
\theoremstyle{plain}
\newtheorem{thmd}{Theorem}
\newtheorem{lemma}{Lemma}
\newtheorem{prop}{Proposition}

\newtheorem{cons}{Corollary}
\theoremstyle{definition}
\newtheorem{deff}{Definition}
\theoremstyle{remark}
	
\newtheorem{remark}{Remark}
\newtheorem{example}{Example}
\everydisplay{\thickmuskip5mu}
\usepackage{thmtools}
\declaretheoremstyle[
notefont=\bfseries,
notebraces={}{},
postheadspace=0pt,
headpunct={},
headformat=\NAME\NOTE,
name = Condition,
style = remark
]{nopar}
\declaretheorem[style=nopar]{cond1}[]

\begin{document}
\thispagestyle{empty}
\sloppy
\pagestyle{plain}
\bibliographystyle{unsrt}
{\textbf{\large{Non-exctinction probability for two branching processes in a joint random environment}}}
 \begin{center}
Dmitrii Arapov, Lomonosov Moscow State University
\\ dmitrii.arapov@math.msu.ru
\end{center}
\section{Abstract}
The paper introduces the model of a pair of branching processes $\{\bol{Z_n} = \left(Z_n^{(1)}, Z_n^{(2)}\right), \; n \in \mathbb{N}_0\}$ in a joint random environment. If the environment is fixed then the sequences $\{Z_n^{(1)}, \; n \in \mathbb{N}_0\}$ and $\{Z_n^{(2)},\; n \in \mathbb{N}_0\}$ are independent branching processes in a varying environment. This model is a particular case of a more general model of a multitype branching process in a random environment. We establish the asymptotic relation $\pr\left(Z_n^{(1)} > 0, Z_n ^{(2)}>0 \right) \sim C n^{-a}$ as $n \to \infty$, where the parameter $a$ depends only on the correlation coefficient $\rho$ of an increment of a two-dimensional associated random walk. 
\section{Introduction}
Let $$\Pi^{(j)} = \left\{p_n^{(j)}\right\}_{n \in \mathbb{N}_0}, \quad f^{(j)}(s) = \sum_{n = 0}^{\infty} p_n^{(j)} s^n, \, j = 1, 2, \ldots, N,$$ where random variables $p_1^{(j)}, p_2^{(j)}, \ldots$ satisfy the relations
$$ p_n^{(j)} \ge 0, \; n \in \mathbb{N}_{0}, \quad \sum_{n = 0}^{\infty} p_n^{(j)} = 1 \text{ a.s.}, \; j = 1, \ldots, N.$$
Let $\bol{\Pi} = \left\{\Pi_1, \Pi_2, \ldots \right\}$, where $\Pi_m = \left( \Pi_m^{(1)}, \ldots, \Pi_m^{(N)}\right)$ is an independent copy of $\left( \Pi^{(1)}, \ldots, \Pi^{(N)}\right)$ for every $m \in \mathbb{N}$. We also define
\begin{equation}
\label{rand_envi}
\bol{f} = \left\{ \left( f_1^{(1)}, \ldots, f_1^{(N)}\right), \left( f_2^{(1)}, \ldots, f_2^{(N)}\right), \ldots \right\},
\end{equation}
where for every $m \in \mathbb{N}$ and $j = 1, \ldots, N$ $f_m^{(j)}$ is the generating function of the sequence $\Pi_m^{(j)}$. The sequence $\bol{f}$ is called a random environment.
The stochastic process $$\left\{\bol{Z_n} = \left(Z_n^{(1)}, \ldots, Z_n^{(N)}\right)\right\}_{n \in \mathbb{N}_0}, \,\,\, Z_0^{(j)} = 1, \; j = 1, \ldots, N,$$ is said to be the group of N branching process in the random environment $\bol{f}$ if in the case of the given $\bol{f}$ the sequences $\left\{Z_n^{(j)}, n \in \mathbb{N}_0\right\}$, $j = 1, \ldots, N$ are independent branching process in the varying environment. The function $f_n^{(j)}$ is the offspring generating function for the $n$-th generation and $j$-th population particle.
\\ \ab A random walk with the increments
$$\ov{X}_i = \left( \ln \left({f}_i^{(1)}\right)'(1), \ldots, \ln \left(f_i^{(N)}\right)'(1) \right)$$
can be associated with the branching process $\bol{Z_n} = \left( Z_n^{(1)}, \ldots, Z_n^{(N)}\right)$.
The process $\{\bol{Z_n}\}$ is said to be critical if the equality $\ex \ov{X}_1 = 0$ holds.
\\ \ab Given N = 1, the stochastic process $\{Z_n\}$ is a branching process in a random environment (BPRE). The model of BPRE was first introduced by Smith and Wilkinson in 1969 in paper \cite{smith}. Note that $\prf(Z_n >0)$ in the case of linear fractional  generating functions $\{f_n\}$ can be explicitly expressed in terms of the associated random walk (see \cite{agr}, page 40).
\\ \ab
The group of N branching processes $\{Z_n\}$ is a particular case of a general multitype model. In a classical multitype model, particles of one type can produce particles of other types. Critical multitype processes were considered in paper \cite{multitype}. In \cite{multitype} the probability of non-extinction of at least one type of particles was studied. This general result is proved under strong assumptions. Our conditions are weaker.
\\ \ab The asymptotic behavior of the probability of non-extinction of critical BPRE in a particular case of linear fractional  generating functions $\{f_n\}$ was first studied by M. V. Kozlov in 1976 in \cite{Kozlov}. After a quarter of a century G. Kersting and I. Geiger obtained the asymptotics for the general case (see \cite{Geiger}). To be more precise, the following asymptotic relation was obtained:
\begin{equation}
\label{1D}
\pr\left(Z_n > 0\right) \sim C \pr \left( S_j > 0, \, j \le n\right) \sim \frac{C\widehat{C}}{\sqrt{n}}, \;\;\; n \to \infty,
\end{equation}
where $C$ and $\widehat{C}$ are some constants. Note that the explicit expression for $\widehat{C}$ can be found for instance in \cite{VPSS}.
\\ \ab Given $N = 2$, the process $\{\bol{Z_n}\}$ will be referred to as a pair of branching processes in a joint random environment (PBPJRE). It is the model that we are mostly interested in. Relation (\ref{1D}) makes us think that the asymptotic behavior of the probability of non-extinction of PBPJRE is connected with the asymptotic behavior of the probability of positivity of the associated random walk. It should be pointed out that the problems concerning the positivity of random walks were investigated in detail by D. Denisov and V. Wachtel in \cite{Den}, \cite{DenRe}. These results are described in Section \ref{to_stay_positive}.
\\ The main result of the present paper is a theorem about non-extinction of critical PBPJRE. Under certain assumptions the following asymptotic relation can be stated:
\begin{equation}
\label{main_equivalent}
\pr\left(Z_n^{(1)} > 0, Z_n^{(2)} > 0\right) \sim \frac{C}{n^{p/2}}, \; n \to \infty,\quad p = \frac{\pi}{ \arccos{(-\rho})},
\end{equation}
where $\rho$ is the correlation coefficient of the increment.
\\ \ab While preparing of the present work we discovered preprint \cite{Pre}, which contains asymptotic relation (\ref{main_equivalent}) for a particular case of a linear fractional  $\{f_k^j\}$. Note that in \cite{Pre} rigid and apparently excessive conditions are imposed. However, the common ideas of this work and preprint \cite{Pre} are quite close. Moreover, some of the results of \cite{Pre} are more general than our results. We will not use the methods and ideas of \cite{Pre}.
\\ \ab
Note that we consider $N = 2$ because in this case the asymptotic relation can be explicitly described. In particular, in this case the parameter p can be found.
\\ \ab This paper is organized as follows. The main result of the present paper is formulated in Section \ref{main_result_section}. Some results from \cite{Den}, \cite{DenRe} devoted to the positivity of multidimensional random walks are presented in Section \ref{to_stay_positive}. Section \ref{Plus_Measure} contains $\pr^{+}$ measure construction, where $\pr^{+}$ is an important tool for investigating the associated random walk behavior. Section \ref{proof_main_nogeom} is devoted to the proof of Theorem \ref{main_nogeom}. Proofs of all auxiliary propositions and lemmas are contained in Section \ref{vspomog}.
\section{Main result}
\label{main_result_section}
We should define the model of a group of branching processes more formally.
\begin{deff}
A stochastic process
$$\left\{\bol{Z_n} = \left(Z_n^{(1)}, \ldots, Z_n^{(N)} \right), \, n \in \mathbb{N}_0 \right\}, \quad Z_0^{(j)} = 1, \; j = 1, \ldots, N$$
is said to be the group of $N$ branching processes in a joint random environment $\bol{f}$ if for every natural number $n$ the following equalities hold:
$$\left. \exf \left( \prod\limits_{j = 1}^{N} t_j^{Z_n^{(j)}}\right|\bol{Z_{n-1}}\right) :=\left. \ex \left( \prod\limits_{j = 1}^{N} t_j^{Z_n^{(j)}}\right|\bol{f}, \bol{Z_{n-1}}\right) = \prod \limits_{j = 1}^N \left(f_n^{(j)} \right)^{Z_{n-1}^{(j)}} \left( t_j\right).$$
\end{deff}
The main result of the present paper is the following theorem about PBPJRE non-extinction.
\begin{thmd}
\label{main_nogeom}
Let $\bol{Z_n} = \left(Z_n^{(1)}, Z_n^{(2)}\right)$ be a critical PBPJRE. Let the sequence of random vectors $\{\ov{S}_n\}$ be the associated random walk with increments $\ov{X}_i = \left(X_i^{(1)}, X_i^{(2)}\right)$, $i \in \mathbb{N}$. Let
\begin{equation}
\label{xidef}
\xi_k^{(l)} := \frac{\left(f_{k+1}^{(l)}\right)''(1)}{\left(\left(f_{k+1}^{(l)}\right)'(1)\right)^2}, \quad l = 1, 2; \; k \in \mathbb{N}_0.
\end{equation}
Suppose that $\ex \left|\ov{X}_{i}\right|^2 \ln \left(1 + \left|\ov{X}_{i}\right|\right) < \infty$, $\mathrm{corr}\left(X_i^{(1)}, X_i^{(2)}\right) = \rho \in (0, 1)$, there exists $q > 0$ such that $\ex \left(\xi_1^{(1)}\right)^q < \infty$,
$\ex \left(\xi_1^{(2)}\right)^q < \infty$. Then 
\begin{equation}
\label{main_nq}
\pr\left(Z_n^{(1)} > 0, Z_n^{(2)} > 0\right) \sim \frac{C}{n^{p/2}}, \; n \to \infty,\quad p = \frac{\pi}{ \arccos{(-\rho})}.
\end{equation}
\end{thmd}
\section{Positivity of multidimensional random walks}
\label{to_stay_positive}
In this Section we will discuss some results from papers \cite{Den} and \cite{DenRe} required for the proof of Theorem \ref{main_nogeom}. First let us introduce some notation.
\subsection{Notation}
We denote the unit sphere in $\mathbb{R}^{d}$ as $\mathbb{S}^{d-1}$.
Throughout the following we use letters $C$, $\widehat{C}$, $\widetilde{C}$ with or without indices to denote constants.
The Euclidean norm of a vector $\ov{v} \in \mathbb{R}^d$ is denoted by $|\ov{v}|$. In other words, if $\ov{v} \in \mathbb{R}^d$, then
$$|\ov{v}| = \left(\sum_{j = 1}^{d}v_j^{2}\right)^{1/2}.$$
For convenience we use the notation $\ov{v} > C$, where $C \in \mathbb{R}$, $\ov{v} \in \mathbb{R}^d$, in the case if inequality $v_i > C$ holds for all $i =1, 2, \ldots, d$. Also by $\ov{w} > \ov{v}$, where $\ov{w}, \ov{v} \in \mathbb{R}^{d}$, we mean that $w_i > v_i$ for all $i =1, 2, \ldots, d$. We use the similar notation for the symbols $<, \ge, \le$.
We use the following notation for the minimum:
$$a \wedge b := \min\left(a, b\right), \quad a, b \in \mathbb{R}.$$
In this paper both ways of writing $\ov{v}$ and ${\bf{v}}$ for the vectors are valid.
By $\|\cdot\|$ throughout the following we mean the operator norm induced by the Euclidean norm, i.e. $\|A\| = \|A\|_2 = \sqrt{\lambda_{max}\left(A^{T} A\right)}$, where $\lambda_{max}(B)$ is the maximum eigenvalue of the matrix B.
\subsection{The Dirichlet problem for Laplace-Beltrami operator}
\begin{deff}
\label{cone_definition}
A set $K \subset \mathbb{R}^d$ is said to be an cone if there exists an open and connected set $\Sigma \subset \mathbb{S}^{d-1}$ such that $K$ is the union of the rays emitted from the origin and passing through the points of the set $\Sigma$.
\end{deff}
\label{Dirichet}
Let $\Sigma \subset \mathbb{S}^{d-1}$ be the set from the previous definition. Let us consider the Dirichlet problem for the Laplace-Beltrami operator on the unit sphere
\begin{align*}
\begin{cases}
\ L_{\mathbb{S}^{d-1}} g\left(\ov{x}\right) = -\lambda g(\ov{x}), &\ov{x} \in \Sigma,
\\
\ g(\ov{x}) = 0, &\ov{x} \in \partial \Sigma.
\\
\end{cases}
\end{align*}
Note that the spectrum is at most countable, discrete, and separated from zero (see \cite{DenRe}).
More precisely, the following chain of inequalities holds:
$$0 < \lambda_1 < \lambda_2 \le \lambda_3 \le \ldots$$
Let
\begin{equation}
\label{pDef}
p:= \sqrt{\lambda_1 +(d/2 - 1)^2} - (d/2 - 1).
\end{equation}
Let $g_1$ be an eigenfunction corresponding to the eigenvalue $\lambda_1$. Let also
\begin{equation}
\label{uDef}
u(\ov{x}) = \left|\ov{x}\right|^{p} g_1 \left( \frac{\ov{x}}{|x|}\right).
\end{equation}
Note that due to the continuity hence boundedness of the function $g_1$ there exists a parameter $C$ such that for every $\ov{x} \in K$ 
\begin{equation}
\label{simpU}
u(\ov{x}) \le C \left|\ov{x}\right|^p.
\end{equation}
We are interested in the following particular case.
\begin{example}
\label{MK_example}
Given $d = 2$, $\Sigma$ is the arc $[-\varphi, \pi/2 + \varphi]$ of the unit circle, where $\varphi \in (0, \pi/4)$.
In this case the parameter $p$ defined by relation (\ref{pDef}) has the following form
$$p = \frac{\pi}{\arccos(-\rho)}, $$
where $\rho = \sin 2\varphi$.
Whence the function 
\begin{align}
\label{def_g1}
g_1(t) =
\begin{cases}
\ (-1)^k \cos {\left(4kt\right)}, &\varphi = -\pi/4 + \pi /(8k), \; k \in \mathbb{N},
\\
\ \sin{\left(p(\varphi + t)\right)}/\cos{\left(p\varphi\right)}, &\text{otherwise}.
\\
\end{cases}
\end{align}
is the eigenfunction corresponds to $\lambda_1$.
\begin{remark}
\label{msign}
Note that the following relations hold
$g_1(t) \ge 0$ if $ t \in \left[-\varphi, \pi/2 + \varphi\right]$ and $g_1(t) = 0$ if and only if $t = -\varphi$ or $t = \pi/2 + \varphi$.
\end{remark}
\end{example}
\subsection{Asymptotic relation for the probability of positivity of random walks with non-correlated components of the increments}
\label{DenSubsection}
Let $\left\{\ov{S}_n\right\}$ be a random walk in $\mathbb{R}^{d}$ with increments $\ov{X}_i = \left(X_i^{(1)}, \ldots, X_i^{(d)} \right)$, $d\in \mathbb{N}$. Let $K \subset \mathbb{R}^d$ be a cone, $\ov{x} \in K$. Let also $\tau(\ov{x})$ be the exit time of a random walk from the cone $K$, i.e.
\begin{equation}
\label{tauDef}
\tau(\ov{x}) = \inf \left\{n \in \mathbb{N}\, :\, \ov{x} + \ov{S}_n \notin K\right\}.
\end{equation}
Throughout the following for every natural number $n$ and $\ov{x} \in K$ we will use notation
\begin{equation}
\label{mn_def}
m_{n}(\ov{x}):= \pr\left( \tau(\ov{x}) > n\right).
\end{equation}
V. Wachtel and D. Denisov in paper \cite{DenRe} impose two group of the assumptions.
\begin{cond1}[{\bf{(G)}}:]
\mbox{}
\label{GCond} 
\begin{enumerate}[label=\subscript{\bf{G}}{{\arabic*}}, labelsep = 0pt]
\item \label{G1}) the cone $K$ is Lipschitz and starlike, i.e. there exists the point $O \in K$ such that for every $A \in K$ the segment $[O, A]$ is the subset of the K;
\item \label{G2}) there exists a constant $C_1$ such that for every vector $\ov{x} \in K$ the following inequality holds
\begin{equation}
\label{ucUP}
u(\ov{x}) \le C_1 \left|\ov{x}\right|^{p - 1} \mathrm{dist} \left( \ov{x}, \partial K \right);
\end{equation}
\item) there exists a constant $C_2$ such that for every vector $\ov{x} \in K$ the following inequality holds
\begin{equation}
\label{ucDO}
u(\ov{x}) \ge C_2 |\ov{x}|^{p - 1} \mathrm{dist} \left( \ov{x}, \partial K \right).
\end{equation} \label{G3}
\end{enumerate}
\end{cond1}
\begin{cond1}[\bf{(M)}:]
\mbox{}
\label{MCond}
\begin{enumerate}[label=\subscript{\bf{M}}{{\arabic*}}, labelsep = 0pt]
\item) $\ex X^{(i)} = 0, \, i = 1, \ldots, d;$ \label{M1}
\item) $\ex \left(X^{(i)}\right)^2 = 1, \, i = 1, \ldots, d$ and $\ex X^{(i)} X^{(j)} = 0$ for $i \neq j$; \label{M2}
\item) $\ex \left|\ov{X}\right|^2 \ln \left(1 + \left|\ov{X}\right|\right) < \infty$. In the case $p>2$ we additionally assume that $\ex \left|\ov{X}\right|^{p} < \infty$. \label{M3}
\end{enumerate}
\end{cond1}
Under the assumptions (G) and (M) the following theorems hold.
\begin{thmd}[\cite{DenRe}]
\label{den1}
Let the assumptions (G) and (M) hold. Then the function
\begin{equation*}
\label{V}
\widehat{V}\left(\ov{x}\right):= \lim_{n \to \infty} \ex u\left(\ov{x} + \ov{S}_n\right) I\left\{ \tau(\ov{x}) > n\right\}
\end{equation*}
is well-defined and harmonic for $\left\{\ov{S}_n\right\}$ killed at leaving $K$, i.e. for every natural number $n$ and every vector $\ov{x} \in K$
\begin{equation*}
\label{V_harm}
\widehat{V}\left(\ov{x}\right)= \ex \widehat{V}\left(\ov{x} + \ov{S}_n\right) I\left\{ \tau(\ov{x}) > n\right\}.
\end{equation*}
Futhermore, if $p \ge 1$, then
\begin{equation*}
\label{VuEq}
\widehat{V}(\ov{x}) \sim u(\ov{x}), \;\; \mathrm{dist}\left( \ov{x}, \partial K\right) \to \infty.
\end{equation*}
\end{thmd}
\begin{thmd}[{\bf{Theorem 3}}, \cite{DenRe}]
\label{den2}
Let the assumptions (G) and (M) hold, $p \ge 1$. Then
\begin{enumerate}
\item[1)] there exists a positive constant $C$ and vector $\ov{x}_0 \in K$ such that for every vector $\ov{x} \in K$ the following inequality holds
\begin{equation*}
\label{uUnEq}
\pr \left( \tau(\ov{x}) > n \right) \le C \frac{u \left( \ov{x} + \ov{x}_0\right)}{n^{p/2}};
\end{equation*}
\item[2)] uniformly in $\ov{x} \in K$ such that $\left|\ov{x}\right| \le \sqrt{n}/\ln n$
\begin{equation*}
\label{equ}
\pr \left( \tau(\ov{x}) > n \right) \sim \frac{\varkappa \widehat{V}(\ov{x})}{n^{p/2}}, \;\; n \to \infty,
\end{equation*}
where $\varkappa > 0$ is a constant, which does not depend on $\ov{x}$.
\end{enumerate}
\end{thmd}
Note that condition \ref{M2} includes that the variance of the random walk increment is finite. However, we do not impose this condition on Theorem \ref{main_nogeom} because finiteness of the variance follows from condition \ref{M3}. We will need another group of conditions, which is the reformulation of conditions \ref{M1}, \ref{M2}, \ref{M3} for the two-dimensional random walk with correlated components of increments. Let $\left\{\ov{S}_n\right\}$ be a two-dimensional random walk with an increment $\ov{X} = \left(X^{(1)}, X^{(2)}\right)$. Let also
$$\rho = \mathrm{corr}\left( X^{(1)}, X^{(2)}\right), \quad p = \frac{\pi}{\arccos(-\rho)}.$$
The random walk $\ov{S}_n$ is said to satisfy a condition (C) if its increment satisfies conditions \ref{C1}, \ref{C2}, \ref{C3}
\begin{enumerate}[label=\subscript{\bf{C}}{{\arabic*}}, labelsep = 0pt]
\item) \label{C1} $\ex X^{(1)} = \ex X^{(2)} = 0;$
\item) \label{C2} $\ex\left|\ov{X}\right|^2 \ln \left(1 + \left|\ov{X}\right|\right) < \infty;$
\item) \label{C3} if $\rho < 0$, then $\ex \left|\ov{X}\right|^p < \infty$.
\end{enumerate}
Note that under Theorem \ref{main_nogeom} the associated random walk of BPRJRE satisfies condition (C).
\subsection{Asymptotic relation for the probability of positivity of random walks with correlated components of the increments}
Let $\left\{\ov{S}_n \right\}$ be a random walk in $\mathbb{R}^2$ with an increment $\ov{X} = \left(X^{(1)}, X^{(2)}\right)$ such that $\ex X^{(1)} = \ex X^{(2)} = 0;$ $\ex \left(X^{(1)}\right)^2 = \sigma_1^2;$ $\ex \left(X^{(2)}\right)^2 = \sigma_2^2$; $\mathrm{corr}\left( X^{(1)}, X^{(2)}\right) = \rho \in (-1, 1).$ In this case the components are correlated. However, there exists a linear transformation which leads to the new cone and the new random walk satisfying condition \ref{M2}.
Let
\begin{equation}
\label{MatrixDef}
M := \frac{1}{\sqrt{1 - \rho^2}}
\begin{pmatrix}
\cos \varphi / \sigma_1 & -\sin \varphi /\sigma_2 \\
-\sin \varphi / \sigma_1 & \cos \varphi/\sigma_2 \\
\end{pmatrix},
\quad \sin 2 \varphi = \rho.
\end{equation}
Then the covariance matrix of $M\ov{X}$ is the identity matrix. The following theorem is the corollary of theorems \ref{den1}, \ref{den2} applied to the transformed random walk.
\begin{thmd}
\label{den_2D}
Let $K = \mathbb{R}^{+} \times \mathbb{R}^{+}$, a two-dimensional random walk $\left\{\ov{S}_n\right\}$ satisfies the condition (C). Let $\tau(\ov{x})$ be the exit time of a random walk from the cone $K$ with starting point $\ov{x} \in K$, as defined in (\ref{tauDef}). Then the following statements hold.
\begin{enumerate}
\item Uniformly in $\ov{x} \in K$ such that $\left|\ov{x}\right| \le \sqrt{n}/\ln n$
\begin{equation}
\label{e}
\pr \left( \tau(\ov{x}) > n \right) \sim \frac{\varkappa V\left(M\ov{x}\right)}{n^{p/2}}, \;\; n \to \infty,
\end{equation}
where $\varkappa$ is some real positive number which does not depend on $\ov{x}$.
\item There exists a real positive number $C$ and a vector $\ov{x}_0 \in MK$ such that for every $\ov{x} \in K$, we have inequality
\begin{equation}
\label{u}
\pr \left( \tau(\ov{x}) > n \right) \le C \frac{u \left( M\ov{x} + \ov{x}_0\right)}{n^{p/2}}.
\end{equation}
\end{enumerate}
In this case the function $V$ for every $\ov{x} \in K$ defined as follows
\begin{equation}
\label{V2}
V\left(M\ov{x}\right):= \lim_{n \to \infty} \ex u\left(M\left(\ov{x} + \ov{S}_n\right)\right) I\left\{ \tau(\ov{x}) > n\right\}.
\end{equation}
Moreover:
\begin{itemize}
\item for every vector $\ov{x} \in K$ and every natural number $n$
\begin{equation}
\label{V_harm2}
V\left(M\ov{x}\right)= \ex V\left(M\left(\ov{x} + \ov{S}_n\right)\right) I\left\{ \tau(\ov{x}) > n\right\};
\end{equation}
\item the following asymptotic relation holds
\begin{equation}
\label{VuEq2}
V(M\ov{x}) \sim u(M\ov{x}), \;\; \mathrm{dist}\left( \ov{x}, \partial K\right) \to \infty.
\end{equation}
\end{itemize}
\end{thmd}
In what follows, we agree to write $\ov{x}_0$ for the parameter from inequality (\ref{u}), $K$ for the the first quadrant, $\tau(\ov{x})$ for the exit time of a random walk with starting point $\ov{x}$ from the quadrant $K$. Throughout the following the functions $V$ and $u$ correspond to the functions $\widehat{V}$ and $u$ defined in Subsection \ref{DenSubsection}, respectively.
\begin{cons}
Suppose for two vectors $\ov{x}, \ov{y} \in K$ that $\ov{x} \le \ov{y}$. Then we have
\begin{equation}
\label{V_mon}
V\left(M\ov{x}\right) \le V\left(M\ov{y}\right).
\end{equation}
\begin{proof}
Note that the inequality $\ov{x} \le \ov{y}$ entails $m_n\left(\ov{x}\right) \le m_n\left(\ov{y}\right)$ for every natural number $n$, where the function $m_n$ is defined by (\ref{mn_def}). From this inequality and relation (\ref{e}), we get the desired inequality.
\end{proof}
\end{cons}
\begin{cons}
There exists a constant $C$ such that for every $\ov{x} \in K$ 
\begin{equation}
\label{VU_un_eq}
V(M\ov{x}) \le C u(M\ov{x} + \ov{x}_0).
\end{equation}
\end{cons}
\begin{proof}[Proof of Theorem \ref{den_2D}]
Let $\varphi := 1/2 \arcsin\rho \in \left(-\pi/4, \pi/4\right).$ Obviously, the image $MK$ of the cone $K = \mathbb{R}^{+}\times \mathbb{R}^{+}$ under the map $M$ is as shown in Figure ~\ref{MKdisp}.
\begin{figure}[H]
\centering
\begin{minipage}{.5\textwidth}
\centering
\includegraphics[width=1.\linewidth, height=200px]{
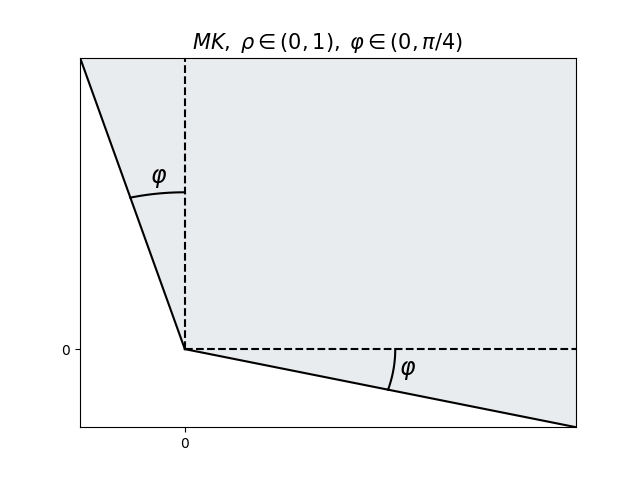}
\end{minipage}%
\begin{minipage}{.5\textwidth}
\centering
\includegraphics[width=1.\linewidth, height=200px]{
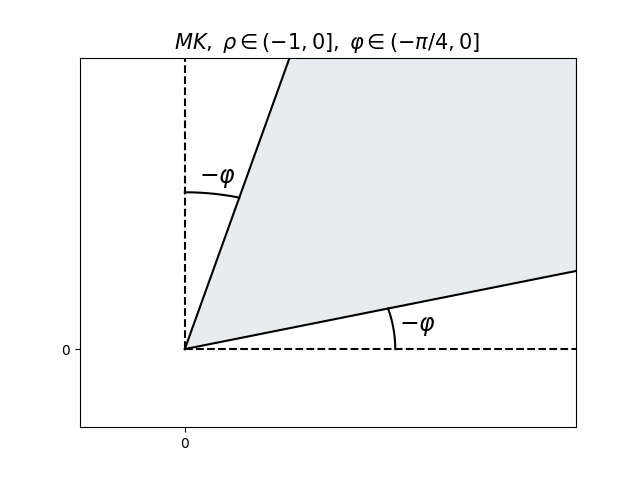}
\end{minipage}
\caption{ $MK$}
\label{MKdisp}
\end{figure}
\begin{remark}
From Figure~\ref{MKdisp} and Example \ref{MK_example} we obtain the following representation for the parameter p which is defined by (\ref{pDef}):
$$p = \frac{\pi}{\arccos\left(-\rho\right)}.$$
Particularly, if $\rho \in (0, 1)$, then $p \in (1, 2)$.
\end{remark}
From Figure~\ref{MKdisp} the validity of condition \ref{G1} for the cone $MK$ obviously follows. In Example \ref{MK_example} we get the description for Laplace-Beltrami operator's eigenfunction $g_1$, which corresponds to the minimum eigenvalue. Explicit representation (\ref{def_g1}) for the function $g_1$ allows us to prove the validity of conditions \ref{G2} and \ref{G3} for the cone $MK$. In other words, following Proposition \ref{G2G3} is true.
\begin{prop}
\label{G2G3}
For every number $\rho \in (-1, 1)$ there exist positive constants $C_1$, $C_2$ such that for every vector $\ov{x} \in MK$
$$C_1 \mathrm{dist}\left(\ov{x}, \partial MK\right) \le \left|\ov{x}\right| g_1 \left( \frac{\ov{x}}{\left|\ov{x}\right|}\right) \le C_2 \mathrm{dist}\left(\ov{x}, \partial MK\right).$$
\end{prop}
\begin{proof}[Proof of Proposition \ref{G2G3}]
Let us rewrite $\ov{x} \in MK$ in the way $r\left(\cos t, \sin t\right)$, where $t \in \left( -\varphi, \pi/2 + \varphi \right)$, $r > 0$, i.e. $t$ and $\varphi$ are the polar angle and the polar radius, respectively. If $x_2 \le x_1$, then
\begin{align*}
\mathrm{dist}\left(\ov{x}, \partial MK \right) = r \sin \left( \varphi + t\right),
\end{align*}
since $\mathrm{dist}\left(\ov{x}, \partial MK \right)$ is the length of the cathetus opposite to the angle $\varphi + t$ in the right triangle with hypotenuse of length $r$. If $\varphi \notin \left\{-\pi/4 + \pi /(8k), \; k\in \mathbb{N}\right\}$, then
$$\left|\ov{x}\right| g_1 \left( \frac{\ov{x}}{\left|\ov{x}\right|}\right) = r \frac{\sin{\left(p(\varphi + t)\right)}}{\cos{p\varphi}}.$$
Let us define
$$h(t) := \frac{\sin{\left(p(\varphi + t)\right)}}{\sin{(\varphi + t)} \cos{p\varphi}}, \quad t \in (-\varphi, \pi/4).$$
Note that the following limits are finite:
$$\lim_{t \to -\varphi_{+}} h(t) = \frac{p}{\cos{(p\varphi)}}, \quad \lim_{t \to \pi/4 -} h(t) = \frac{\sin\left(p(\pi/4 + \varphi)\right)}{\sin\left(\pi/4 + \varphi\right) \cos {(p\varphi)}} < \infty.$$
The function $h$ can be continuously extended to the set $\left[-\varphi, \pi/4\right]$. Therefore, $h$ is bounded on $\left[-\varphi, \pi/4\right]$. Furthermore, since $h(t) > 0$ for $t$ from the interval $\left[-\varphi, \pi/4\right]$, the lower bound constant can be chosen positive.
\\ \ab Consider the case $\varphi = -\pi/4 + \pi /(8k)$ for some $k \in \mathbb{N}$. Then
$$\left|\ov{x}\right| g_1 \left( \frac{\ov{x}}{\left|\ov{x}\right|}\right) = (-1)^k r \cos{(4kt)}.$$
The following relations hold
$$\lim_{t \to -\varphi_{+}} \frac{(-1)^k \cos{(4kt)}}{\sin{(\varphi + t)}} = 4k, \quad \lim_{t \to \pi/4_{-}} \frac{(-1)^k \cos{(4kt)}}{\sin{(\varphi + t)}} = \frac{1}{\sin(\varphi + \pi/4)} < \infty; \; k \in \mathbb{N}.$$
Thus, the function $$h(t) := \frac{(-1)^k \cos{(4kt)}}{\sin{(\varphi + t)}} > 0, \;\;\; t\in\left(-\varphi, \pi/4\right),$$
which is continuously extended to the points$-\varphi, \pi/4$ are bounded and positive on the set $[-\varphi, \pi/4]$.
We show that if $x_2 \le x_1$, then there exists desired constants $C_1$, $C_2$. Taking into account the symmetry (see Figure \ref{MKdisp}) such constants can be also found in the case $x_2 > x_1$. It proves Proposition \ref{G2G3}.
\end{proof}
Taking into account Proposition \ref{G2G3}, we conclude that conditions \ref{G2} and \ref{G3} hold. Thus, assumptions G are true. For the random walk $\left\{ M\ov{S}_n\right\}$ conditions \ref{M1} and \ref{M2} hold by the definition of the matrix $M$. Let us consider assumption \ref{M3}.
\begin{prop}
\label{use_uneq}
For every $\ov{z} \in K$ the following inequalities hold
\begin{equation}
\label{eq_im_M}
|M\ov{z}| \ge \frac{|\ov{z}|}{\|M^{-1}\|};
\end{equation}
\begin{equation}
\label{dist_eq_1}
\mathrm{dist}\left(\ov{z}, \partial K\right) \le \|M^{-1}\| \mathrm{dist}\left(M\ov{z}, \partial MK\right).
\end{equation}
\end{prop}
\begin{proof}[Proof of Proposition \ref{use_uneq}]
Inequality (\ref{eq_im_M}) follows from the definition of operator norm. More precisely, for every $\ov{z} \in K$
$|M^{-1}M \ov{z}| \le \|M^{-1}\| |M\ov{z}|.$ Relation (\ref{dist_eq_1}) is also a simple corollary of the definition of operator norm:
\begin{align*}
\mathrm{dist}(\ov{z}, \partial K) &= \inf_{y \in \partial K} |\ov{z}- \ov{y}| =
\inf_{y \in \partial K} \left|M^{-1}(M\ov{z} - M{\ov{y}}) \right| \le \|M^{-1}\| \inf_{y \in \partial K} |M\ov{z} - M\ov{y}| \le
\\ & \le \|M^{-1}\| \mathrm{dist} \left( M\ov{z}, M\partial K\right) = \|M^{-1}\| \mathrm{dist} \left( M\ov{z}, \partial MK\right).
\end{align*}
\end{proof}
\begin{prop}
\label{lneq}
Let $\ov{X}$ be a two-dimensional random vector. Then the conditions
\begin{align*}
&\ex \left|\ov{X}\right|^2 \ln \left(1 + \left|\ov{X}\right|\right) < \infty, \quad \ex \left|M\ov{X}\right|^2 \ln \left(1 + \left|M\ov{X}\right|\right) < \infty
\end{align*}
are equivalent.
\end{prop}
\begin{proof}[Proof of Proposition \ref{lneq}]
Let us note that for every $a > 0$ and $x\ge 0$ 
$$\min \{a, 1\}\ln \left(1 + x\right) \le \ln (1 + ax) \le \max \{a, 1\} \ln(1 +x).$$
Combining this with inequality (\ref{eq_im_M}), we get the statement of Proposition \ref{lneq}.
\end{proof}
Proposition \ref{lneq} lets us conclude that condition \ref{M3} for the random walk $\left\{M\ov{S}_n\right\}$ is valid. Taking into account Theorems \ref{den1}, \ref{den2} and Propositions \ref{G2G3}, \ref{lneq}, we yield Theorem \ref{den_2D}.
\end{proof}
\begin{prop}
For every vector $\ov{y}$ from the closure of the cone $MK$ and for every $\delta > 0$ there exists parameter $C = C\left(\ov{y}, \delta\right)$ such that for every $\ov{x} \in MK$ the inequality $\mathrm{dist}\left(\ov{x}, \partial MK\right) > \delta$ implies that
\begin{equation}
\label{uNPart}
u \left( \ov{x} + \ov{y}\right) \le C u\left(\ov{x}\right).
\end{equation}
Furthermore, the parameter $C$ can be chosen such that
\begin{equation*}
\label{exp_for_constant}
C(\ov{y}, \delta) = \widetilde{C}\left(1 +\left\|M\right\| \left\|M^{-1}\right\|\frac{|\ov{y}|}{\delta} \right)^{p-1}\left(1 + \frac{\dist\left(\ov{y}, \partial MK \right)}{\delta}\right) \le \widetilde{C}\left(1 + \frac{\|M\| \|M^{-1}\|}{\delta}\right)^p \mathrm{max}\left(|\ov{y}|^p, 1\right),
\end{equation*}
where $\widetilde{C}$ does not depend on both $\ov{y}$ and $\delta$.
\end{prop}
\begin{proof}
Firstly, we fix $\ov{y} \in MK$ and $\delta > 0$. Since inequalities (\ref{eq_im_M}), (\ref{dist_eq_1}) are valid for every $\ov{z} \in MK$ the relation
$$\mathrm{dist}\left( \ov{z}, \partial MK\right) \le \|M\|\|M^{-1}\| \left|\ov{z}\right|$$
holds. Hence, taking into account (\ref{ucUP}), (\ref{ucDO}), $\mathrm{dist}\left(\ov{x}, \partial MK\right) > \delta$, we get 
\begin{align*}
u \left( \ov{x} + \ov{y}\right) &\le C_1 \left|\ov{x} + \ov{y} \right|^{p - 1} \mathrm{dist}\left(\ov{x} + \ov{y}, \partial MK\right) \le
\\ &\le C_1 |\ov{x}|^{p-1} \dist(\ov{x}, \partial MK) \left(1 + \|M\|\|M^{-1}\|\frac{|\ov{y}|}{\delta} \right)^{p-1} \left( 1 + \frac{\dist(\ov{y}, \partial MK)}{\dist(\ov{x}, \partial MK)}\right) \le
\\ &\le \widetilde{C}\left(1 + \|M\|\|M^{-1}\|\frac{|\ov{y}|}{\delta} \right)^{p-1}\left(1 + \frac{\dist\left(\ov{y}, \partial MK \right)}{\delta}\right) u(\ov{x}) = C(\ov{y}, \delta) u(\ov{x}).
\end{align*}
Therefore, the following inequalities hold:
$$C\left(\ov{y}, \delta\right) \le \widetilde{C} \left(1 + \|M\|\|M^{-1}\|\frac{|\ov{y}|}{\delta} \right)^{p} \le \widetilde{C}\left(1 + \frac{\|M\| \|M^{-1}\|}{\delta}\right)^p \mathrm{max}\left(|\ov{y}|^p, 1\right).$$
From the continuously of the functions $u$ and $C(\cdot, \delta)$ on the closure of the cone $MK$ we conclude that relation (\ref{uNPart}) is true for every $\ov{y}$ from the closure of $MK$.
\end{proof}
\section{P-plus measure}
\label{Plus_Measure}
In the one-dimensional case so-called $\pr^{+}$ measure plays crucial role in investigations of branching processes (see, e.g. \cite{Vatutin}, Section 5.2 ``Changes of measures''). We introduce the similar construction in the two-dimensional case. Note that most of this Section is similar to preprint \cite{Pre}.
\\ \ab Let $K = \mathbb{R}^{+}\times \mathbb{R}^{+}$. Let also $\ov{X}$ has the same distribution as the increment of the associated random walk $\left\{\ov{S}_n\right\}$. Let $\varphi: \; \mathbb{R}^{2} \rightarrow \mathbb{R}$ be a function such that
\begin{enumerate}[label=\subscript{\bf{H}}{{\arabic*}}, labelsep = 0pt]
\item. $\varphi\left(\ov{x}\right) > 0$ for $\ov{x} \in K$; \label{H_1}
\item. $\varphi\left(\ov{x}\right) = 0$ for $\ov{x} \notin K$; \label{H_2}
\item. ${\bf{E}}\, \varphi \left( \ov{x} + \ov{X}\right) = {\bf{E}}\, \varphi \left( \ov{x} + \ov{X}\right) I\{ \ov{x} + \ov{X} \in K \} = \varphi \left(\ov{x}\right).$ \label{H_3}
\end{enumerate}
Then for every $n \in \mathbb{N}$ and $\ov{x} \in K$ the expression
$$\pr_{n, \bfx}^{+}(A) := {\bf{E}} \frac{\varphi \left( \ov{x} + \ov{S}_n\right)}{\varphi \left( \ov{x}\right)}
I\{\ov{S}_j + \ov{x} > 0, \, j \le n \} I_{A}$$
defines the probability measure on the $\sigma$-field $\mathcal{F}_n = \sigma \left( \Pi_1^{(1)}, \Pi_1^{(2)}, \ldots, \Pi_n^{(1)}, \Pi_n^{(2)}\right),$ where for every $i \in \mathbb{N}$ and $k \in \left\{ 1, 2\right\}$ $\Pi_i^{(k)}$ is the sequence from the definition of the random environment $\bol{f}$ (see equality (\ref{rand_envi})).
Given $\ov{x}$, the measures $\{\pr_{n, \bfx}^{+}\}$ are consistent. Therefore, there exists the probability measure $\pr_{\bfx}^{+}$ on the $\sigma$-field $\mathcal{F} = \sigma \left( \Pi_n^{(1)}, \Pi_n^{(2)}; \; n \in \mathbb{N} \right)$ such that $\left.\pr_{\bfx}^{+}\right|_{\mathcal{F}_n} = \pr_{n, \bfx}^{+}$.
\begin{remark}
\label{cont_PP_plus}
As follows from the definition of the $\pr_{\bfx}^{+}$ measure, $\pr_{\bfx}^{+}$ is absolutely continuous with respect to the measure $\pr$ on every $\sigma$-field $\mathcal{F}_k$, $k \in \mathbb{N}$.
\end{remark}
Consider $\varphi(\ov{x}) := V(M\ov{x}) I\{ x_1 > 0, x_2 > 0\},$ where $\ov{x} = (x_1, x_2) \in \mathbb{R}^{2}$, the function $V$ is defined by (\ref{V2}). This choice of the function $\varphi$ can violate condition \ref{H_1} since $V(M\ov{x})$ may be zero. However, taking into account relation (\ref{VuEq2}) and positivity of the function $u$, the condition is valid for all vectors $\ov{x}$ which are far enough from the border of the quadrant $K$. We define the measure $\pr_{\bfx}^{+}$ only for $\ov{x}$ from the set
$$\mathcal{X}_V = \left\{\ov{y} \in K : V\left(M\ov{y}\right) \neq 0 \right\}.$$
Now we are ready to prove the two-dimensional analogue of Lemma 5.2 from \cite{Vatutin}.
\begin{thmd}
\label{th_plus}
Let $\ov{x} \in \mathcal{X}_V$. Assume that the following conditions are valid for a sequence of random variables $\left\{Y_k\right\}$:
\begin{enumerate}
\item[1)] the random variable $Y_k$ is $\mathcal{F}_k$-measurable for every natural $k$;
\item[2)] $\left\{Y_n\right\}$ converges as $n \to \infty$ to some random variable $Y_{\infty}$ $\pr_{\bfx}^{+}$-a.s.;
\item[3)] $\left\{Y_n\right\}$ is a uniformly bounded sequence, it means that there exists a positive constant $C$ such that $|Y_k| \le C$ for any natural $k$.
\end{enumerate}
Then
$$\lim_{n \to \infty} {\bf{E}} \left( Y_n \left.\right| \tau(\ov{x}) > n \right) = {\bf{E}}_{\bfx}^{+} Y_{\infty}.$$
\end{thmd}
\begin{proof}
We fix natural numbers $k$ and $n$, $n > k$ and a vector $\ov{x} \in \mathcal{X}_V$.
The following equalities hold:
\begin{align*}
\ex \left( Y_k | \tau\left(\ov{x}\right) > n\right) &= \frac{\ex Y_k I\{\tau\left(\ov{x}\right) > n\}}{m_n\left(\ov{x}\right)}=
\frac{ \ex \left[ Y_k I\{\tau\left(\ov{x}\right) > k\} {\bf{E}} \left( I\{\ov{S}_j + \ov{x} > 0, \, j \in (k, n]\} \left.\right| \mathcal{F}_k \right) \right]}{m_n\left(\ov{x}\right)} =
\\&=\frac{ \ex \left[Y_k I\{\tau\left(\ov{x}\right) > k\} m_{n - k}\left(\ov{S}_k + \ov{x} \right)\right]}{m_n\left(\ov{x}\right)}.
\end{align*}
Let us rewrite asymptotic relation (\ref{e}) in terms of $m_n$:
\begin{equation*}
\label{simTh1}
m_n(\ov{x}) \sim \frac{\varkappa V\left(M\ov{x}\right)}{n^{p/2}}, \;\;\; n \to \infty.
\end{equation*}
Hence, the convergence
\begin{equation*}
\label{as}
\frac{m_{n-k} \left(\ov{S}_k + \ov{x}\right)}{m_n(\ov{x})} \to
\frac{V\left(M\left(\ov{S}_k + \ov{x}\right)\right)}{V\left(M \ov{x}\right)}, \quad n \to \infty
\end{equation*}
takes place. Moreover, taking into account inequalities (\ref{simpU}), (\ref{u}), uniform boundness of the sequence $\{Y_k\}$ and relation (\ref{e}), we get:
$$\left| \frac{Y_k I\{\tau(\ov{x}) > k\} m_{n-k}\left( \ov{S}_k +\ov{x}\right)}{m_n(\ov{x})}\right| \le C \frac{ m_{n-k}\left( \ov{S}_k +\ov{x}\right)}{m_n(\ov{x})}
\le C_1 \frac{\left|M\left(\ov{S}_k + \ov{x}\right) + \ov{x}_0\right|^p}{V\left(M\ov{x}\right)}.$$
Using the triangle inequality and the Minkowski inequality, we have:
\begin{equation*}
\label{mink}
\left(\ex \left|M\left(\ov{S}_k + \ov{x}\right) + \ov{x}_0\right|^p\right)^{1/p} \le \left(\ex \left|M\ov{S}_k\right|^{p}\right)^{1/p} + \left|M\ov{x}\right| +|\ov{x}_0| \le k \|M\| \left({\bf{E}}\left|\ov{X}\right|^p\right)^{1/p} + \left|M\ov{x}\right| + \left|\ov{x}_0\right|< \infty,
\end{equation*}
where $\ov{X}$ is a random vector distributed as the increment of the associated random walk $\left\{\ov{S}_n\right\}$.
Therefore, by Lebesque's dominated convergence theorem the relation
\begin{equation}
\label{TH1_eyk_lim}
\lim_{n \to \infty} \ex \left( Y_k | \tau(\ov{x}) > n\right) = \ex Y_k \frac{V\left( M(\ov{x} + \ov{S}_k)\right)}{V\left( M\ov{x}\right)} I\{\tau(\ov{x}) > k\} = {\bf{E}}_{\bfx}^{+}Y_k
\end{equation}
is valid.
\\ \ab Now we would like to prove that for every $q \in (1, +\infty)$ the equality
\begin{equation}
\label{Th1KL}
\lim_{k \to \infty} \varlimsup_{n \to \infty} {\bf{E}}\left. \left( \left|Y_n - Y_k\right| I\{M\left(\ov{S}_n + \ov{x}\right)\in K_L \} \right| \tau(\ov{x}) > qn \right) = 0
\end{equation}
holds. Let us denote $d(\ov{y}):= \mathrm{dist}(\ov{y}, \partial MK)$ for $\ov{y} \in MK$. By (\ref{VuEq2}) and (\ref{uNPart}) there exists a positive real number $L$ such that the inequality $d(\ov{y}) \ge L$ implies
\begin{equation}
\label{KLUneq}
\frac{u\left(\ov{y} + \ov{x}_0\right)}{V(\ov{y})} \le \frac{C_2 u(\ov{y})}{V(\ov{y})} \le 2C_2 =: C_3,
\end{equation}
where the parameter $C_3$ does not depend on $\ov{y}$. We fix $L$ and denote $K_{L}:= \{\ov{y} \in MK: d(\ov{y}) \ge L\}$, $U_L := MK \backslash K_L$. Using inequality (\ref{u}) and asymptotic relation (\ref{e}), for sufficiently large $n$ we have:
\begin{align}
&\ex \left. \left( \left| Y_n - Y_k\right| I\left\{M\left(\ov{S}_n + \ov{x}\right) \in K_L \right\} \right| \tau(\ov{x}) > qn \right) = \nonumber
\\ &=
\frac{1}{m_{qn}(\ov{x})}\ex \left[|Y_n - Y_k|m_{qn - n}(\ov{S}_n + \ov{x}) I \left\{\tau(\ov{x}) > n, M\left(\ov{S}_n + \ov{x}\right) \in K_L\right\}\right] \le \nonumber
\\ &\le
C_4 \left(\frac{q}{q-1}\right)^{p/2} \ex \left[ \left|Y_n - Y_k\right| I\left\{\tau(\ov{x}) > n, \, M\left(\ov{S}_n+\ov{x}\right) \in K_L\right\} \frac{u\left(M\left(\ov{S}_n + \ov{x}\right) +\ov{x}_0 \right)}{V(M\ov{x})}\right] = \nonumber
\\ & =
C_4 \left(\frac{q}{q-1}\right)^{p/2} \ex_{\bfx}^{+} \left[ \left|Y_n - Y_k\right| I\left\{M\left(\ov{S}_n + \ov{x}\right) \in K_L\right\}
\frac{u\left(M\left(\ov{S}_n + \ov{x}\right) + \ov{x}_0 \right)}{V\left(M\left(\ov{S}_n + \ov{x}\right)\right)}\right] \label{KLsoot}.
\end{align}
By inequality (\ref{KLUneq}), the estimate
$$\frac{u\left(M\left(\ov{S}_n + \ov{x}\right) + \ov{x}_0\right)}{V\left(M\left(\ov{S}_n + \ov{x}\right)\right)} \le C_3$$
holds on the event $\left\{M\left(\ov{S}_n + \ov{x}\right) \in K_L\right\}$.
From this estimate and relation (\ref{KLsoot}), we get
\begin{equation*}
\label{TH1KL_uneq}
{\bf{E}}\left. \left( \left|Y_n - Y_k\right| I\{M\left(\ov{S}_n + \ov{x}\right) \in K_L \} \right| \tau(\ov{x}) > qn \right) \le C_3 C_4 \ex_{\bfx}^{+} \left|Y_n - Y_k\right|.
\end{equation*}
\ab By uniform boundness and ${\bf{P}}_{\bfx}^{+}$-a.s. convergence of the sequence $\{Y_k\}$, we obtain
$$ \lim_{n \to \infty} {\bf{E}}_{\bfx}^{+} |Y_n - Y_k| = {\bf{E}}_{\bfx}^{+} |Y_{\infty} - Y_k|, \;\;\; \lim_{k \to \infty}{\bf{E}}_{\bfx}^{+} |Y_{\infty} - Y_k| = 0.$$
Therefore,
$$\varlimsup_{k \to \infty}\varlimsup_{n \to \infty}{\bf{E}}\left. \left( \left| Y_n - Y_k\right| I\{M\left(\ov{S}_n + \ov{x}\right) \in K_L \} \right| \tau(\ov{x}) > qn \right) \le C_3 C_4 \varlimsup_{k \to \infty} \varlimsup_{n \to \infty} \ex_{\bfx}^{+} |Y_n - Y_k| = 0.$$
This proves relation (\ref{Th1KL}).
\\ \ab Our next aim is to show that for every $q \in (1, +\infty)$ and every $k\in \mathbb{N}$ the equality
\begin{equation}
\label{Th1UL}
\lim_{n \to \infty} {\bf{E}}\left. \left( \left|Y_n - Y_k\right| I\left\{M\left(\ov{S}_n + \ov{x}\right)\in U_L \right\} \right|\tau(\ov{x}) > qn \right) = 0
\end{equation}
holds. If $\ov{z} \in M^{-1}U_L$, then from inequality (\ref{dist_eq_1}) the inequality $\mathrm{dist}(\ov{z}, \partial K) \le \|M^{-1}\| L$ follows.
Let
\begin{align*}
T: = \|M^{-1}\| L + 1, \quad D_L := \left\{ \ov{z} \in K: \mathrm{dist}(\ov{z}, \partial K) < T \right\},
\quad A_1 := \{\ov{z} \in K: z_1 \ge z_{2}\}, \quad A_2 := \ov{A_1}.
\end{align*}
Note that $M^{-1}U_L \subseteq D_L$. Since the sequence $\left\{Y_k\right\}$ is uniformly bounded, the inequality
\begin{equation}
\label{TH1UnUL}
\ex \left( \left|Y_n - Y_k\right| I\{M \left(\ov{S}_n + \ov{x}\right) \in U_L, \tau(\ov{x}) > qn\} \right) \le 2C \pr \left( \tau(\ov{x}) > qn, \ov{S}_n + \ov{x} \in M^{-1}U_L\right)
\end{equation}
holds. Furthermore,
\begin{align*}
{\bf{P}}\left( \tau(\ov{x}) > qn, \ov{S}_n + \ov{x} \in M^{-1}U_L\right) &\le \int\limits_{D_L} {\bf{P}} \left( \tau(\ov{x}) > qn, \ov{S}_n + \ov{x} \in d\ov{y}\right) = 
\\ &= \int\limits_{D_L} {\bf{P}} \left( \tau(\ov{x}) > n, \ov{S}_n + \ov{x} \in d\ov{y}\right) {\bf{P}} \left( \tau(\ov{y}) > qn -n\right). 
\end{align*}
We split the last integral into two parts, which are correspond to integration regions $D_{L} \cap A_i$, $i = 1, 2$.
For every $i \in \{1, 2\}$ we have:
\begin{align*}
&\int\limits_{D_L \cap A_i} \pr \left( \tau(\ov{x}) > n, \ov{x} + \ov{S}_n \in d\ov{y}\right) \pr \left( \tau(\ov{y}) > qn-n\right) \le
\\&\le {\bf{P}}\left( \tau(\ov{x}) > n, \ov{x} + \ov{S}_n \in D_L \cap A_i\right) {\bf{P}} \left( S_j^{(3 - i)} > -T, j \le qn - n\right) \le
\\ &\le {\bf{P}}\left( \tau(\ov{x}) > n\right) {\bf{P}} \left( S_j^{(3-i)} > -T, j \le qn - n \right).
\end{align*}
Hence, taking into account inequality (\ref{TH1UnUL}), we get the following estimate:
\begin{align*}
& \ex \left( \left|Y_n - Y_k\right| I \left\{M(\ov{x} + \ov{S}_n) \in U_L, \tau(\ov{x}) > qn\right\} \right) \le
\\&\le 2C \pr\left( \tau(\ov{x}) > n\right) \left( \pr \left( S_j^{(1)} > -T, j \le qn - n\right) + \pr \left( S_j^{(2)} > -T, j \le qn - n\right)\right).
\end{align*}
Thus, (\ref{Th1UL}) follows from the relations
\begin{align*}
\lim_{n \to \infty} \frac{{\bf{P}}\left( \tau(\ov{x}) > n\right)}{\pr \left( \tau(\ov{x}) > qn\right)} = q^{p/2}, \quad \lim_{n \to \infty}{\bf{P}} \left( S_j^{(i)} > -T, j \le qn - n\right) = 0, \quad i = 1, 2.
\end{align*}
\ab From (\ref{Th1KL}) and (\ref{Th1UL}) we get
\begin{equation}
\label{limTH1}
\lim_{k \to \infty} \varlimsup_{n \to \infty} \left|{\bf{E}}\left. \left( Y_n - Y_k \right| \tau(\ov{x}) > qn \right) \right| = 0
\end{equation}
for every $q \in (1, +\infty)$. We fix $q > 1$ and $k < n$. Using the triangle inequality twice, we get
\begin{align}
&m_n(\ov{x}) \left| \ex \left( Y_n \left.\right| \tau (\ov{x}) > n \right) - \ex_{\bfx}^{+} Y_{\infty} \right| \le \left|\ex Y_n I \left\{\tau(\ov{x})> qn\right\} - \pr \left(\tau(\ov{x}) > qn\right)\ex_{\bfx}^{+} Y_{\infty}\right| + \nonumber
\\&+ \left|\ex Y_nI\{ n < \tau(\ov{x})\le qn\}\right| + \pr \left(n < \tau(\ov{x}) \le qn\right) \left|\ex_{\bfx}^{+} Y_{\infty}\right|. \label{fin0}
\end{align}
The estimate
\begin{align*}
&\left|\ex Y_n I \left\{\tau(\ov{x})> qn\right\} - \pr \left(\tau(\ov{x}) > qn\right)\ex_{\bfx}^{+} Y_{\infty}\right| \le
\\ &\le m_{qn}(\ov{x})\left. \left| \ex\left(Y_n - Y_k \right| \tau(\ov{x})> qn \right) \right| + m_{qn}(\ov{x}) \left. \left| \ex\left(Y_k \right| \tau(\ov{x}) > qn \right) - \ex_{\bfx}^{+}Y_{\infty} \right|
\end{align*}
holds.
Hence, taking into account (\ref{TH1_eyk_lim}), (\ref{limTH1}), we get
\begin{align*}
&\varlimsup_{n \to \infty}\frac{1}{m_n(\ov{x})}\left|\ex Y_n I \left\{\tau(\ov{x})> qn\right\} - \pr \left(\tau(\ov{x}) > qn\right)\ex_{\bfx}^{+} Y_{\infty}\right| \le
\\ &\le \lim_{k \to \infty} \varlimsup_{n \to \infty} \frac{m_{qn}(\ov{x})}{m_n(\ov{x})} \left|{\bf{E}}\left. \left( Y_n - Y_k \right| \tau(\ov{x}) > qn \right) \right| + \lim_{k \to \infty} \varlimsup_{n \to \infty} \frac{m_{qn}(\ov{x})}{m_n(\ov{x})}\left. \left| \ex\left(Y_k \right| \tau(\ov{x}) > qn \right) - \ex_{\bfx}^{+}Y_{\infty} \right| =0.
\end{align*}
Thus,
\begin{equation}
\label{fin1}
\lim_{n \to \infty}\frac{1}{m_n(\ov{x})}\left|\ex Y_n I \left\{\tau(\ov{x})> qn\right\} - \pr \left(\tau(\ov{x}) > qn\right)\ex_{\bfx}^{+} Y_{\infty}\right| = 0.
\end{equation}
The inequalities
\begin{align}
\left|\ex Y_nI\{ n < \tau(\ov{x})\le qn\}\right| \le C \pr\left(n < \tau(\ov{x}) \le qn\right), \quad \left|\ex_{\bfx}^{+} Y_{\infty}\right| \le 2C \label{fin2}
\end{align}
take place. Note that
$$\lim_{n \to \infty} \frac{\pr \left( n < \tau(\ov{x}) \le qn \right)}{m_n(\ov{x})} = \lim_{n \to \infty} \left(1 - \frac{m_{qn}(\ov{x})}{m_n(\ov{x})}\right) = 1 - q^{-p/2}.$$
Hence, from relations (\ref{fin0}), (\ref{fin1}), (\ref{fin2}) it follows that for every $q > 1$
\begin{equation*}
\label{fin}
\varlimsup_{n \to \infty} \left| {\bf{E}} \left( Y_n \left.\right| \tau(\ov{x}) > n \right) - {\bf{E}}_{x}^{+} Y_{\infty} \right| \le 3C \left(1 - q^{-p/2}\right).
\end{equation*}
Tending $q \downarrow 1$ in the right-hand side of the last inequality, we prove Theorem \ref{th_plus}.
\end{proof}
\section{Proof of Theorem \ref{main_nogeom}}
\label{proof_main_nogeom}
To prove Theorem \ref{main_nogeom} we need the following auxiliary statements.
\begin{lemma}
\label{mainlemma}
Let sequences $\{a_n\}_{n \in \mathbb{N}}$, $\left\{\delta_{n, m}\right\}_{n, m \in \mathbb{N}}$, $\left\{b_{n, m}\right\}_{n, m \in \mathbb{N}}$ be such that the equality $a_n = \delta_{n, m} + b_{n, m}$ takes place for every natural $n, m$. Assume that the following conditions are valid:
\begin{enumerate}
\item[1)] $\delta_{n, m} \ge 0$, $b_{n, m} \ge 0$ for $n$, $m \in \mathbb{N}$;
\item[2)] $\displaystyle \varliminf\limits_{m \to \infty} \varlimsup\limits_{n \to \infty} \delta_{n, m} = 0$;
\item[3)] for every natural $m$ there exists the finite limit $\displaystyle \lim_{n \to \infty} b_{n, m} =: b_{m}$.
\end{enumerate}
Then $a:= \displaystyle \lim_{n \to \infty} a_n = \lim_{m \to \infty} b_{m}< \infty.$ Moreover, if $b_m > 0$ for some $m$, then $a > 0$.
\end{lemma}
\begin{lemma}
\label{smallNuneq}
Let $\left\{\ov{Z}_n\right\}$ be a PBPJRE. Assume that the associated random walk satisfies the condition (C). Then for every $\varepsilon > 0$ there exists a positive number $T$ such that for any $t > T$  and any $n \in \mathbb{N}$ the inequality
$${\bf P}\left( \ov{Z}_n > 0, \, \tau(\ov{t}) \le n\right) \le \frac{\varepsilon}{n^{p/2}}$$
holds, where $\ov{t} = (t, t).$
\end{lemma}
\begin{lemma}
\label{lemma7}
Under the conditions of Theorem \ref{main_nogeom} for every $\ov{x} \in \mathcal{X}_V$ there exists finite limit
\begin{equation*}
\label{lemma7lim}
\lim_{n \to \infty} \left.\pr \left( \ov{Z}_n > 0 \right.|\tau(\ov{x}) > n \right):= r(\ov{x}).
\end{equation*}
\end{lemma}
For the convenience of the reader, proofs of Lemmas \ref{mainlemma}, \ref{smallNuneq}, \ref{lemma7} are moved into Section \ref{vspomog}.
\begin{proof}[Proof of Theorem \ref{main_nogeom}]
We choose a vector $\ov{t} = (t, t)$ sufficiently large to satisfy the inequality $V\left(M\ov{t}\right) > 0$. In other words, $\ov{t} \in \mathcal{X}_V$. Let us denote
\begin{align*}
&\delta_n(\ov{t}) = n^{p/2} \pr \left(\ov{Z}_n > 0, \tau(\ov{t}) \le n\right);
\\ &b_n(\ov{t}) = \pr\left(\ov{Z}_n > 0| \tau(\ov{t}) > n\right) \cdot n^{p/2} \pr(\tau(\ov{t}) > n).
\end{align*}
By Lemma \ref{smallNuneq} for every $\varepsilon > 0$ there exists a positive parameter $t_0$ such that for any $t > t_0$ the inequality $\displaystyle \varlimsup_{n \to \infty} \delta_n(\ov{t}) < \varepsilon$ holds. Hence,
$$\lim_{t \to \infty} \varlimsup_{n \to \infty} \delta_n(\ov{t}) = 0.$$
From asymptotic relation (\ref{e}) and Lemma \ref{lemma7} it follows that there exists finite limit
$$\lim_{n \to \infty} b_n(\ov{t}) = \varkappa r(\ov{t}) V\left(M\ov{t}\right) =: b(\ov{t}).$$
The function $b(\ov{t})$ is positive for sufficiently large $t$. Using the expression
$$n^{p/2}\pr\left(\ov{Z}_n > 0\right) = \delta_n(\ov{t}) + b_n(\ov{t})$$
and Lemma \ref{mainlemma}, we conclude existence, finiteness and positivity of the limit
$$\lim_{n \to \infty} n^{p/2}\pr\left(\ov{Z}_n > 0\right).$$
Theorem \ref{main_nogeom} is proved.
\end{proof}
\section{Auxiliary statements}
\label{vspomog}
\subsection{Proofs of Lemmas \ref{mainlemma} and \ref{smallNuneq} }
\begin{proof}[Proof of Lemma \ref{mainlemma}]
Note that the upper limit of the sequence $\{a_n\}$ can not be equal to $+\infty$. Indeed, there exists a natural number $m$ such that $\varlimsup_{n \to \infty} \delta_{n, m} < +\infty$. Thus,
$$\varlimsup_{n \to \infty} a_n \le \varlimsup_{n \to \infty} \delta_{n, m} + \lim_{n \to \infty} b_{n, m} = \varlimsup_{n \to \infty} \delta_{n, m} + b_m < +\infty.$$
By the inequality $\delta_{n, m} \ge 0$ for $n, m \in \mathbb{N}$, for every $m$ we have:
\begin{equation*}
\label{b_and_varliminf_a}
b_{m} \le \varliminf_{n \to \infty} a_n.
\end{equation*}
From the last inequality it follows that the upper limit of the sequence $\{b_m\}$ is finite. Trivially, if there exists a natural $m$ such that $b_m > 0$, then the lower limit of the sequence $\{a_n\}$ is positive.
Furthermore,
$$\varlimsup_{m \to \infty} b_m \le \varliminf_{n \to \infty} a_n \le \varlimsup_{n \to \infty} a_n \le \varliminf_{m \to \infty} \varlimsup_{n \to \infty} \delta_{n, m} + \varliminf_{m \to \infty} b_m = \varliminf_{m \to \infty} b_m.$$
Hence,
$$\varliminf_{n \to \infty} a_n = \varlimsup_{n \to \infty} a_n = \varlimsup_{m \to \infty} b_m = \varliminf_{m \to \infty} b_m.$$
Lemma \ref{mainlemma} is proved.
\end{proof}
\begin{proof}[Proof of Lemma \ref{smallNuneq}]
\label{first_lemma_uneq_2}
We need the following estimate (see, e.g. \cite{Afan}, p. 161)
\begin{equation}
\label{survVPSS}
\prf \left( Z_n^{(i)} > 0\right) \le \mathrm{exp}\left(\min_{0\le j\le n}S_j^{(i)}\right), \quad i = 1, 2.
\end{equation}
We fix $\ov{t} = (t, t)$, $t > 0$ and natural $n$. Note that
$$\left\{\tau(\ov{t}) \le n\right\} = \bigsqcup_{m = 1}^{\infty} \left\{\tau \left( \left( m+1\right) \ov{t}\right) > n, \, \tau \left(m\ov{t}\right) \le n \right\}.$$
Thus,
\begin{equation}
\label{F_form}
{\bf P} \left(\ov{Z}_n > 0, \, \tau(\ov{t}) \le n\right) =
\sum_{m = 1}^{\infty} {\bf{P}} \left(\ov{Z}_n > 0,\, \tau \left( \left( m+1\right) \ov{t}\right) > n, \, \tau \left(m\ov{t}\right) \le n \right).
\end{equation}
By the properties of conditional expectation, for every $m \in \mathbb{N}$ we get:
\begin{align*}
\pr \left(\ov{Z}_n > 0,\, \tau \left(\left(m+1\right) \ov{t}\right) > n, \,\tau\left(m\ov{t}\right) \le n \right) &= \ex \exf I \left \{\ov{Z}_n > 0,\, \tau \left( \left(m+1\right) \ov{t}\right) > n, \, \tau \left(m\ov{t}\right) \le n \right\} =
\\ &= \ex \left[ \prf \left( \ov{Z}_n > 0\right) I\left\{\tau \left( \left(m+1\right) \ov{t}\right) > n, \, \tau\left(m\ov{t}\right) \le n\right \}\right].
\end{align*}
Taking into account (\ref{survVPSS}), we have:
\begin{align*}
& \prf \left( \ov{Z}_n > 0\right) = \prf \left( Z_n^{(1)} > 0\right) \prf \left( Z_n^{(2)} > 0\right) \le
\mathrm{exp}\left( \min_{\text{$j \le n$}} S_j^{(1)} + \min_{\text{$j \le n$}} S_j^{(2)}\right).
\end{align*}
Therefore,
\begin{align*}
&
\pr \left(\ov{Z}_n > 0,\, \tau \left(\left(m+1\right) \ov{t}\right) > n, \, \tau\left(m\ov{t}\right) \le n \right) =
\ex \left[ \prf \left( \ov{Z}_n > 0\right) I \left \{\tau \left(\left(m+1\right) \ov{t}\right) > n, \, \tau\left(m\ov{t}\right) \le n\right\} \right] \le
\\
& \le \ex \, \mathrm{exp}\left( \min_{\text{$j \le n$}} S_j^{(1)} + \min_{\text{$j \le n$}} S_j^{(2)}\right) I \left\{\tau \left(\left(m+1\right) \ov{t}\right) > n, \, \tau\left(m\ov{t}\right) \le n\right\} \le
\\
&\le \mathrm{exp}(-mt) \pr \left( \tau\left(\left(m+1\right) \ov{t}\right) > n, \, \tau\left(m\ov{t}\right) \le n \right)
\le \mathrm{exp}(-mt) \pr \left( \tau\left(\left(m+1\right) \ov{t}\right) > n \right).
\end{align*}
Thus, the inequality
\begin{equation}
\label{first_lemma_uneq_1}
\pr \left(\ov{Z}_n > 0,\, \tau\left(\left(m+1\right) \ov{t}\right) > n, \, \tau\left(m\ov{t}\right) \le n \right) \le \mathrm{exp}(-mt) {\bf{P}} \left( \tau\left(\left(m+1\right) \ov{t}\right) > n \right)
\end{equation}
holds. From (\ref{simpU}) and (\ref{u}) it follows that
\begin{equation}
\label{first_lemma_uneq_2}
\pr \left( \tau\left(\left(m+1\right) \ov{t}\right) > n \right) \le C\frac{u\left( (m+1)M\ov{t} + \ov{x}_0\right)}{n^{p/2}}
\le \widehat{C} \frac{\left(m+1\right)^p t^p}{n^{p/2}}.
\end{equation}
Note that $e^{-t}t^p$ is the monotonically decreasing function for $t >p$. Let $t \ge T > p$, then for any natural number $m$ the estimate
$$e^{-mt} t^{p} \le e^{-mT} T^m$$
takes place. Hence, taking into equality (\ref{F_form}) and inequalities (\ref{first_lemma_uneq_1}), (\ref{first_lemma_uneq_2}) for $t \ge T > p$, we get the following relation
\begin{equation*}
\label{first_lemma_main_en_1}
{\bf P}\left( \ov{Z}_n > 0, \, \tau(\ov{t}) \le n\right) \le \widehat{C}\frac{T^p e ^{-T}}{n^{p/2}}\sum_{m = 0}^{\infty} e^{-mT} \left(m+2\right)^p \le \widehat{C}\frac{T^p e ^{-T}}{n^{p/2}}\sum_{m = 0}^{\infty} e^{-mp} \left(m+2\right)^p.
\end{equation*}
The series in the right-hand side converges by Cauchy root test. We fix $\varepsilon > 0$ and take sufficiently large $T$ to satisfy the inequality
$$\widehat{C}T^p \mathrm{exp}\left(-T\right) \sum_{m = 0}^{\infty} e^{-mp} \left(m+2\right)^p < \varepsilon.$$
Lemma \ref{smallNuneq} is proved.
\end{proof}
\subsection{Auxiliary facts about random walks}
Proof of Lemma \ref{lemma7} is much more complicated than those of Lemmas \ref{mainlemma} and \ref{smallNuneq}.
We need some auxiliary statements to prove Lemma \ref{lemma7}.
\begin{prop}
\label{lnn}
Let $\{S_n\}$ be a random walk with increments $X_j$, $\ex X_j = 0$, $\ex X_j^2 = \sigma^2 \in (0, +\infty)$. Then for any $\alpha > 0$ there exists a parameter $C$ such that for every sufficiently large natural $n$ the inequality
\begin{equation*}
\label{lnUneq}
\pr \left( S_j < \alpha \ln n, j \le n \right) \le \frac{C\ln n}{\sqrt n}
\end{equation*}
holds.
\end{prop}
\begin{proof}
See \cite{Vatutin}, Lemma 4.4.
\end{proof}
\begin{prop}[\bf{the concentration inequality}]
\label{cont}
Let $\{S_n\}$ be a random walk with increments $X_j$, $\ex X_j = 0$, $\ex X_j^2 = \sigma^2 \in (0, +\infty)$. Then there exists a real number $C$ such that for any $\Delta > 0$ there exists a natural $N$ such that for every $n > N$
\begin{equation}
\label{conc}
\sup_{x\in \mathbb{R}}\pr\left(S_n \in [x, x + \Delta)\right) \le \frac{C \Delta}{\sqrt{n}}.
\end{equation}
\begin{proof}
See, e.g. \cite{Petrov}, Section \RomanNumeralCaps{3}, Theorem 9.
\end{proof}
\end{prop}
\subsubsection{Localization of a random walk}
\begin{lemma}
\label{comp}
Assume that a two-dimensional random walk $\left\{ \ov{S}_n\right\}$ satisfies condition (C). Let also the correlation coefficient of components of the random walk increments $\rho$ belongs to the interval $(0, 1)$.
Then for any $\varepsilon >0$ there exists a square $D = [\delta, A] \times [\delta, A]$ such that for any $\ov{x} \in \mathcal{X}_V$ the following inequality holds:
\begin{equation}
\label{not_in_sq_p_plus}
\varlimsup_{T \to \infty} \pr_{\bfx}^{+}\left(\ov{S}_T + \ov{x} \notin \sqrt{T}D \right) \le \varepsilon \frac{u\left( \ov{x}_0 + M\ov{x}\right)}{V\left(M\ov{x}\right)}.
\end{equation}
\end{lemma}
\begin{proof}[Proof of Lemma \ref{comp}]
Firstly, we would like to prove the following statement.
\begin{prop}
\label{prop_un_eq}
There exists a parameter $C$ such that for any natural $T$ and $\ov{x} \in K$ the estimate
\begin{equation}
\label{eq_for_exp}
\ex \left[\tau(\ov{x}) \wedge T\right] \le C T^{1 - p/2} u\left(\ov{x}_0 + M\ov{x}\right)
\end{equation}
holds.
\end{prop}
\begin{proof}
Indeed, using inequality (\ref{u}), we get the desired estimation:
\begin{align*}
\ex \left[ \tau(\ov{x}) \wedge T\right] &= \sum_{k = 0}^{+\infty} \pr \left( \tau(\ov{x}) \wedge T > k\right) = \sum_{k = 0}^{T - 1} \pr \left( \tau(\ov{x}) > k\right) \le 1 + Cu\left(\ov{x}_0 + M \ov{x}\right)\sum_{k = 1}^{T-1} \frac{1}{k^{p/2}} \le
\\ &\le \widehat{C} T^{1-p/2} u\left(\ov{x}_0 + M \ov{x}\right).
\end{align*}
\end{proof}
We fix $\ov{x} \in \mathcal{X}_V$. Let us denote for any $0 < \delta < A$ and natural $T$
$$F := \left\{\ov{z}\in K:\, |\ov{z}| \ge A\sqrt{T}\right\}, \quad U:=\left\{\ov{z}\in K: \, |\ov{z}| \le A \sqrt{T}, \, \mathrm{dist}\left( \ov{z}, \partial K\right) < \delta \sqrt{T}\right\}, \quad D := [\delta, A] \times [\delta, A].$$
Note that $\mathbb{R}^2 \backslash \sqrt{T}D \subseteq F \cup U.$ Thus, the inequality
\begin{align}\pr_{\bfx}^{+}\left(\ov{S}_T + \ov{x} \notin \sqrt{T}D \right) \le \pr_{\bfx}^{+}\left(\ov{S}_T + \ov{x} \in U \right) + \pr_{\bfx}^{+}\left(\ov{S}_T + \ov{x} \in F \right) \label{TD_UF}
\end{align}
holds. Let us note that $\ov{S}_{T} + \ov{x} \le \ov{A} \sqrt{T}= \sqrt{T} \left(A, A\right)$ on the event $\left\{\ov{S}_{T} + \ov{x} \in U\right\}$. Therefore,
\begin{align}
V\left(M\ov{x}\right)\pr_{\bfx}^{+}\left(\ov{S}_T + \ov{x} \in U \right) \le V\left(\sqrt{T}M\ov{A}\right) \pr\left(\tau(\ov{x})> T, \ov{S}_T + \ov{x} \in U\right). \label{p_plus_U_set}
\end{align}
We may consider that $A > 1$.
Combining inequality (\ref{u}) and asymptotic relation (\ref{e}), we have
\begin{align}
V\left(\sqrt{T}M\ov{A}\right) &\le C_1 u\left( \ov{x}_0 + \sqrt{T}M\ov{A}\right) \le C_2 \left|R\ov{x}_0 + \sqrt{T} M\ov{A}\right|^{p} \le C_2 \|M\| \left(2T\right)^{p/2} A^{p} \left|1 + \frac{\ov{x}_0}{\sqrt{T}\left|M\ov{A}\right|}\right|^p \le \nonumber
\\&\le C_3 T^{p/2} A^{p} \left| 1 + \left\| M^{-1}\right\| \frac{\ov{x}_0}{\sqrt{2}} \right|^{p} = C_4 T^{p/2} A^{p}, \label{VTMA}
\end{align}
where the parameter $C_4$ does not depend on both $T$ and $A$. Let us denote
$$\Gamma_i := \left\{\ov{z} = (z_1, z_2) \in K: z_i \le \delta \sqrt{T}\right\}, \quad i = 1, 2.$$
For $i \in \{1, 2\}$ the following relations are valid:
\begin{align}
&\pr\left( \tau(\ov{x}) > T, \ov{S_T} + \ov{x} \in \Gamma_i \right) \le \nonumber
\\&\le \int\limits_{-x_1}^{+\infty} \int\limits_{-x_2}^{+\infty} \pr\left( \ov{S}_j + \ov{x} \in d\ov{w}, \, j \le T/2\right) \pr \left( \ov{S}_j + \ov{w} > 0, j \le T/2, S_T^{(i)} + w_i \in \left(0, \delta \sqrt{T}\right) \right) \le \nonumber
\\ &\le \int\limits_{-x_1}^{+\infty} \int\limits_{-x_2}^{+\infty} \pr\left( \ov{S}_j + \ov{x} \in d\ov{w}, \, j \le T/2\right) \pr \left(S_T^{(i)} + w_i \in \left(0, \delta \sqrt{T}\right) \right). \label{Gamma_i}
\end{align}
By inequality (\ref{conc}), we get
\begin{align*}
\pr \left(S_T^{(i)} + w_i \in \left(0, \delta \sqrt{T}\right) \right) \le \sum_{k = 1}^{\left[\delta \sqrt{T}\right] + 1} \pr\left( S_T^{(i)} + w_i \in [k, k-1)\right) \le C\frac{\delta\sqrt{T} + 1}{\sqrt{T}} = C\delta + \frac{C}{\sqrt{T}}.
\end{align*}
Hence, taking into account inequalities (\ref{u}) and (\ref{Gamma_i}), we have
\begin{align*}
\pr\left( \tau(\ov{x}) > T, \ov{S_T} + \ov{x} \in \Gamma_i \right) \le C_5 \frac{u\left(\ov{x}_0 + M \ov{x}\right)}{T^{p/2}} \left(\delta + \frac{1}{\sqrt{T}} \right).
\end{align*}
Whence, by (\ref{p_plus_U_set}) and (\ref{VTMA}), it follows that
\begin{align}
\varlimsup_{T \to \infty} \pr_{\bfx}^{+}\left(\ov{S}_T + \ov{x} \in U \right) &\le \varlimsup_{T \to \infty} \sum_{i = 1}^{2} \frac{V\left( \sqrt{T}M\ov{A}\right)}{V\left(M\ov{x}\right)} \pr\left( \tau(\ov{x}) > T, \ov{S_T} + \ov{x} \in \Gamma_i \right) \le \nonumber
\\ &\le 2C_4C_5\frac{u\left(\ov{x}_0 + M \ov{x}\right)}{V\left(M\ov{x}\right)}A^{p} \delta = C_6 \frac{u\left(\ov{x}_0 + M \ov{x}\right)}{V\left(M\ov{x}\right)}A^{p} \delta, \label{p_plus_U_final}
\end{align}
where $C_6$ does not depend on both $\delta$ and $A$.
\\ \ab By inequalities (\ref{simpU}) and (\ref{VU_un_eq}), we get the estimation
\begin{align*}
V\left(M\ov{x}\right)\pr_{\bfx}^{+}\left(\ov{S}_T + \ov{x} \in F \right) \le C_2 \ex \left|\ov{x}_0 + M \left(\ov{S}_T + \ov{x} \right)\right|^{p} I\left\{\tau(\ov{x}) > n, \ov{S}_T + \ov{x} \in F\right\}.
\end{align*}
On the event $\left\{\ov{S}_T + \ov{x} \in F\right\}$ the following inequalities
\begin{align*}
\left|\ov{x}_0 + M \left(\ov{S}_T + \ov{x} \right)\right|^{p} &\le \|M\| \left|\ov{S}_T + \ov{x}\right|^{p} \left|\frac{ \ov{x}_0}{\left|M\left(\ov{S}_T + \ov{x}\right)\right|} + 1 \right|^p \le \|M\| \left|\ov{S}_T + \ov{x}\right|^{p} \left|\frac{\|M^{-1}\| \ov{x}_0}{A \sqrt{T}} + 1\right|^p \le
\\ &\le C_7 \left|\ov{S}_T + \ov{x}\right|^p \le C_7 \left( A\sqrt{T}\right)^{p-2} \left|\ov{S_T} + \ov{x}\right|^2
\end{align*}
are valid. The last inequality in this chain takes place because $p < 2$.
Thus, we get the estimation
\begin{align}
V\left(M\ov{x}\right)\pr_{\bfx}^{+}\left(\ov{S}_T + \ov{x} \in F \right) &\le C_2 C_7 \left(A\sqrt{T}\right)^{p-2}\ex \left|\ov{S}_T + \ov{x}\right|^2 I\left\{\tau(\ov{x})> n, \ov{S}_T + \ov{x} \in F\right\}\le \nonumber
\\ & \le C_8 \left(A \sqrt{T}\right)^{p-2} \ex \left|\ov{S}_{\tau(\ov{x}) \wedge T} + \ov{x} \right|^2.\label{p_plus_F_set}
\end{align}
Note that for any natural $n$
$$\left. \ex \left( \left|\ov{x} + \ov{S}_n\right|^2 \right| \mathcal{F}_{n-1} \right) = \left|\ov{x} + \ov{S}_{n-1}\right|^2 + \ex \left|\ov{X}\right|^2,\quad \mathcal{F}_n = \sigma \left(\ov{S}_j, \; j \le n\right).$$
Therefore, the sequence $\{\chi_n(\ov{x}) = \left|\ov{x} + \ov{S}_n\right|^2 - n \mu\}_n$ is a martingale with respect to the filtration $\left(\mathcal{F}_n\right)$, where $\mu = \ex \left|\ov{X}\right|^2$. Moreover, $\tau(\ov{x}) \wedge T$ is a bounded stopping time. Therefore, by the optimal stopping theorem, we get
\begin{equation}
\label{stop}
\ex \left|\ov{x} + \ov{S}_{\tau(\ov{x}) \wedge T}\right|^2 = \left|\ov{x}\right|^2 + \mu \ex \left[\tau(\ov{x}) \wedge T\right].
\end{equation}
By Proposition \ref{prop_un_eq}, we have
$$\mu \ex \left[\tau(\ov{x}) \wedge T\right] \le C T^{1 - p/2} u\left(\ov{x}_0 + M\ov{x}\right). $$
Hence, taking into account (\ref{p_plus_F_set}), (\ref{stop}), we get the inequality
$$\pr_{\bfx}^{+}\left(\ov{S}_T + \ov{x} \in F \right) \le C_9 \left|\ov{x}\right|^2\frac{A^{p-2} T^{p/2 - 1}}{V\left(M\ov{x}\right)} + C_{10}A^{p-2}\frac{u\left(\ov{x}_0 + M\ov{x}\right)}{V\left(M\ov{x}\right)}.$$
Thus, the relation
\begin{align}
\varlimsup_{T \to \infty}\pr_{\bfx}^{+}\left(\ov{S}_T + \ov{x} \in F \right) \le C_{10}A^{p-2}\frac{u\left(\ov{x}_0 + M\ov{x}\right)}{V\left(M\ov{x}\right)} \label{p_plus_F_final}
\end{align}
is valid.
Finally, combining inequalities (\ref{TD_UF}), (\ref{p_plus_U_final}), (\ref{p_plus_F_final}), we get 
\begin{align*}
\varlimsup_{T \to \infty} \pr_{\bfx}^{+}\left(\ov{S}_T + \ov{x} \notin \sqrt{T}D \right) &\le \varlimsup_{T \to \infty} \pr_{\bfx}^{+}\left(\ov{S}_T + \ov{x} \in U \right) + \varlimsup_{T \to \infty} \pr_{\bfx}^{+}\left(\ov{S}_T + \ov{x} \in F \right) \le
\\ &\le \frac{u\left(\ov{x}_0 + M \ov{x} \right)}{V\left(M\ov{x}\right)} \left( C_6 A^p \delta + C_{10} A^{p-2}\right).
\end{align*}
We fix $\varepsilon > 0$, then fix $A > 0$ such that the inequality $C_{10} A^{p-2} \le \varepsilon / 2$ holds. It is possible because $p<2$.
Then we fix $\delta > 0$ such that $C_6 \delta A^{p-1}\le \varepsilon / 2.$ Therefore,
\begin{equation*}
\label{lemma2_final_eq}
\varlimsup_{T \to \infty} \pr_{\bfx}^{+}\left(\ov{S}_T + \ov{x} \notin \sqrt{T}D \right) \le \varepsilon \frac{u\left(\ov{x}_0 + M\ov{x}\right)}{V\left(M\ov{x}\right)}.
\end{equation*}
Lemma \ref{comp} is proved.
\end{proof}
\subsubsection{Separation of a random walk from the boundary of the quadrant}
\begin{lemma}
\label{lemmaU}
Let condition (C) be valid for a two-dimensional random walk $\left\{ \ov{S}_n\right\}$. Let $D \subset K$ be a measurable bounded set such that $\delta := \mathrm{dist}\left(D, \partial K\right) > 0$. Then for any $s \in (0, 1/2)$ and $\ov{x} \in \mathcal{X}_V$ 
$$\lim_{T \to \infty}\pr_{\bfx}^{+}\left( \ov{S}_T + \ov{x} \in \sqrt{T}D, \exists j > T: \ov{S}_j + \ov{x} \notin U_{T, s}\right) = 0,$$
 where
\begin{equation}
\label{utsDef}
U_{T, s} = \left\{ \ov{x} \in K : \mathrm{dist}\left(\ov{x}, \partial K\right) > T^{1/2 - s}\right\}.
\end{equation}
\end{lemma}
\begin{proof}
By Theorem \ref{th_plus} we get
\begin{align*}
&\pr_{\bfx}^{+}\left( \ov{S}_T + \ov{x} \in \sqrt{T}D, \exists j > T: \ov{S}_j + \ov{x} \notin U_{T, s}\right) =
\\ &= \lim_{n \to \infty}\left. \pr \left(\ov{S}_T + \ov{x} \in \sqrt{T}D,\, \exists j \in (T, n]: \ov{S}_j + \ov{x} \notin U_{T, s}\right.|\tau(\ov{x}) > n\right).
\end{align*}
Thus, if we establish the equality
$$\lim_{T \to \infty} \lim_{n \to \infty} \pr \left( \tau(\ov{x}) > n, \, \ov{S}_T + \ov{x} \in \sqrt{T}D,\, \exists j \in (T, n]: \ov{S}_j + \ov{x} \notin U_{T, s}\right) /m_n(\ov{x}) = 0,$$
then Lemma \ref{lemmaU} will be proved.
By the properties of conditional expectation, we have the following equalities
\begin{align*}
&\pr \left( \tau(\ov{x}) > n, \, \ov{S}_T + \ov{x} \in \sqrt{T}D,\, \exists j \in (T, n]: \ov{S}_j + \ov{x}\notin U_{T, s}\right) =
\\&= \pr \left( \tau(\ov{x}) > n, \, \ov{S}_T + \ov{x} \in \sqrt{T}D\right) - \pr \left( \tau(\ov{x}) > n, \, \ov{S}_T + \ov{x} \in \sqrt{T}D,\, \ov{S}_j + \ov{x} \in U_{T, s}, j \in (T, n] \right) =
\\&= \ex I\{\tau(\ov{x}) > T\} I\{\ov{S}_T + \ov{x} \in \sqrt{T}D\} \left(m_{n-T}\left(\ov{S}_T + \ov{x}\right) - m_{n-T}\left( \ov{S}_T + \ov{x} - \ov{w}_T\right)\right),
\end{align*}
where $\ov{w}_T = T^{1/2 - s}\ov{1}$, $\ov{1}:= (1, 1)$, i.e. $\ov{w}_T + K = U_{T, s}.$
By Lebesque's dominated convergence theorem, the following equalities takes place:
\begin{align*}
&\lim_{n \to \infty} \ex I\{\tau(\ov{x}) > T\} I\{\ov{S}_T + \ov{x} \in \sqrt{T}D\} \frac{m_{n-T}\left(\ov{S}_T + \ov{x}\right)}{m_n(\ov{x})} = \ex_{\bfx}^{+} I\{\ov{S}_T + \ov{x} \in \sqrt{T}D\}, 
\\&\lim_{n \to \infty} \ex I\{\tau(\ov{x}) > T\} I\{\ov{S}_T + \ov{x} \in \sqrt{T}D\} \frac{m_{n-T}\left(\ov{S}_T + \ov{x} - \ov{w}_T\right)}{m_n(\ov{x})} = \nonumber
\\&= \ex_{\bfx}^{+} I\{\ov{S}_T + \ov{x} \in \sqrt{T}D\} \frac{V\left(M\left(\ov{S}_T + \ov{x} - \ov{w}_T\right)\right)}{V\left(M\left(\ov{S}_T + \ov{x}\right)\right)}. 
\end{align*}
Existence of the integrable majorants follows by the same arguments as in the proof of Theorem \ref{th_plus}, see inequality (\ref{mink}).
On the event $\left\{ \ov{S}_T + \ov{x} \in \sqrt{T}D\right \}$ from inequalities (\ref{eq_im_M}), (\ref{dist_eq_1}) it follows that
\begin{align*}
&\left|M(\ov{S}_T + \ov{x})\right| \ge \frac{\left|\ov{S}_T + \ov{x}\right|}{\|M^{-1}\|}\ge \frac{\mathrm{dist}( \ov{S}_T +\ov{x}, \partial K)}{\|M^{-1}\|} \ge \frac{\delta \sqrt{T}}{\|M^{-1}\|} \to \infty, \;\; T \to \infty,
\\&\frac{\left| M\ov{w}_T \right|}{\left|M\left(\ov{S}_T + \ov{x}\right)\right|} \le \frac{\|M^{-1}\| \|M\| \sqrt{2} \cdot T^{1/2 - s}}{\delta \sqrt{T}} \le \frac{\|M^{-1}\| \|M\| \sqrt{2}}{\delta T^{s}} \to 0, \;\; T \to \infty, 
\\& \mathrm{dist}\left(M\left(\ov{S}_T +\ov{x}- \ov{w}_T\right), \partial MK\right) 	\ge \frac{\delta \sqrt{T} - T^{1/2 - s}}{\|M^{-1}\|} \to \infty, \;\; T \to \infty. 
\end{align*}
Hence, from equality (\ref{uDef}) and inequality (\ref{VuEq2}), we get as $T \to \infty$:
\begin{equation}
\label{eqVV}
\frac{V\left(M\left(\ov{S}_T +\ov{x} - \ov{w}_T\right)\right)}{V\left(M\left(\ov{S}_T + \ov{x}\right)\right)} \sim \frac{g_1\left(M\left(\ov{S}_T + \ov{x} - \ov{w}_T\right)/|M(\ov{S}_T+ \ov{x} - \ov{w}_T)|\right)}{g_1\left(M\left(\ov{S}_T + \ov{x}\right)/ |M(\ov{S}_T + \ov{x})|\right)},
\end{equation}
where the function $g_1$ is defined by (\ref{def_g1}).
Let us denote for any natural $T$ and $\ov{x} \in \mathcal{X}_V$
\begin{align*}
\ov{A}_T\left(\ov{x}\right) := \frac{M\left(\ov{S}_T +\ov{x} - \ov{w}_T\right)}{\left|M\left(\ov{S}_T +\ov{x}- \ov{w}_T\right)\right|}, \quad \ov{B}_T\left(\ov{x}\right) := \frac{M\left(\ov{S}_T +\ov{x}\right)} {\left|M\left(\ov{S}_T +\ov{x}\right)\right|}.
\end{align*}
On the event $\left\{\ov{S}_T + \ov{x} \in \sqrt{T}D\right\}$ we have the estimation
\begin{align}
\left|\ov{A}_T\left(\ov{x}\right) - \ov{B}_T\left(\ov{x}\right) \right| &\le \frac{\left|M\ov{w}_T\right|}{\left|M\left(\ov{S}_T + \ov{x} - \ov{w}_T\right)\right|} + \left|M\left(\ov{S}_T +\ov{x}\right)\right|\left| \frac{1}{\left|M\left(\ov{S}_T +\ov{x}- \ov{w}_T\right)\right|} - \frac{1}{\left|M\left(\ov{S}_T +\ov{x}\right)\right|}\right| \le \nonumber
\\ &\le \frac{2\left|M\ov{w}_T\right|}{\left|M\left(\ov{S}_T + \ov{x} - \ov{w}_T\right)\right|} \le
\frac{2\sqrt{2}\left\|M\right\| \left\|M^{-1}\right\| T^{1/2 -s}}{\delta \sqrt{T} - T^{1/2 - s}}, \label{convAB}
\end{align}
where the right-hand part tends to zero as $T \to \infty$.
\\ \ab The function $g_1$ is continuous on the set $\ov{\Sigma}$. Therefore, $g_1$ is a uniformly continuous function on $\ov{\Sigma}$. Combining this fact with estimation (\ref{convAB}), taking into account that this estimation is uniform with respect to $\ov{x}$, we have that for any $\varepsilon > 0$ there exists a natural number $T_0$ such that for any $T > T_0$ and $\ov{x} \in \mathcal{X}_V$ on the event $\left\{ \ov{S}_T + \ov{x} \in \sqrt{T}D\right\}$ the following inequality
$$\left|g_1\left(\ov{A}_T\left(\ov{x}\right)\right) - g_1\left(\ov{B}_T\left(\ov{x}\right)\right)\right| < \varepsilon$$
takes place.
\\ \ab Since the set $D$ is bounded, without loss of generality, let D be the subset of the sphere of radius $r$ centered at the origin. Let us recall that $g_1$ vanishes only on the set $\partial \Sigma$ (see Remark \ref{msign}). At the same time the random variable $\ov{B}_T\left(\ov{x}\right)$ is separated from the border $\partial \Sigma$. Indeed,
\begin{align*}
\mathrm{dist}\left(\ov{B}_T \left(\ov{x}\right), \partial \Sigma\right) \ge \frac{\mathrm{dist}\left(\ov{S}_T + \ov{x}, \partial K\right)}{\|M^{-1}\| \left|M\left(\ov{S}_T + \ov{x}\right)\right|} \ge \frac{\delta \sqrt{T}}{\|M^{-1}\| \|M\| r \sqrt{T}} \ge
\frac{\delta }{\|M^{-1}\| \|M\| r} > 0.
\end{align*}
\ab The separation $\ov{B}_T$ from the border $\partial \Sigma$ implies that there exists $\theta > 0$ such that for any $\ov{x}$ the estimation $g_1\left(\ov{B}_T\left(\ov{x}\right)\right) > \theta > 0$ holds. Therefore, for any $\varepsilon >0$ there exists a natural number $T_0$ such that for any $T > T_0$ and any $\ov{x}$ on the event $\left\{\ov{S}_T + \ov{x} \in \sqrt{T}D\right\}$ the inequality
$$\left| 1- \frac{g_1\left(\ov{A}_T\left(\ov{x}\right)\right)}{g_1\left(\ov{B}_T\left(\ov{x}\right)\right)} \right| < \varepsilon$$
takes place.
\\ \ab We fix $\varepsilon > 0$ and number $T_0$ such that for any $T > T_0$ on the event $\left\{\ov{S}_T + \ov{x} \in \sqrt{T}D\right\}$ the following inequalities
\begin{align*}
\left|1 -\frac{V\left(M\left(\ov{S}_T + \ov{x} - \ov{w}_T\right)\right)}{V\left(M\left(\ov{S}_T +\ov{x}\right)\right)} \right| \le \left| 1- \frac{g_1\left(\ov{A}_T\left(\ov{x}\right)\right)}{g_1\left(\ov{B}_T\left(\ov{x}\right)\right)} \right| + \varepsilon \frac{g_1\left(\ov{A}_T\left(\ov{x}\right)\right)}{g_1\left(\ov{B}_T\left(\ov{x}\right)\right)} \le \varepsilon + \varepsilon (1+\varepsilon)
\end{align*}
holds. Then we have
\begin{align*}
&\lim_{n \to \infty} \frac{\pr \left( \tau(\ov{x}) > n, \, \ov{S}_T + \ov{x} \in \sqrt{T}D,\, \exists j \in (T, n]: \ov{S}_j + \ov{x} \notin U_{T, s}\right)}{\pr \left(\tau(\ov{x}) > n\right)} \le
\\&\le \ex_{\bfx}^{+} I\{\ov{S}_T + \ov{x} \in \sqrt{T}D\} \left|1 - \frac{V\left(M\left( \ov{S}_T + \ov{x} - \ov{w}_T\right)\right)}{V\left(M\left(\ov{S}_T + \ov{x}\right)\right)} \right| \le \left(2\varepsilon + \varepsilon^2\right) \ex_{\bfx}^{+} I\{\ov{S}_T + \ov{x} \in \sqrt{T}D\} \le 2\varepsilon + \varepsilon^2.
\end{align*}
Thus,
\begin{align*}
\varlimsup_{T \to \infty}\lim_{n \to \infty} \frac{\pr \left( \tau(\ov{x}) > n, \, \ov{S}_T + \ov{x} \in \sqrt{T}D,\, \exists j \in (T, n]: \ov{S}_j + \ov{x} \notin U_{T, s}\right)}{\pr \left(\tau(\ov{x}) > n\right)} \le 2\varepsilon + \varepsilon^2.
\end{align*}
Tending $\varepsilon \downarrow 0$ in the last inequality, we obtain Lemma \ref{lemmaU}.
\end{proof}
\begin{prop}
Let condition (C) be valid for a two-dimensional random walk $\left\{ \ov{S}_n\right\}$.
Then for any $\alpha > 0$ there exists a positive parameter $C = C(\alpha)$ such that for any sufficiently large $n$ and any
$\ov{x} \in K$:
\begin{equation}
\label{Kn_uneq}
\pr\left(\tau(\ov{x})> n, \ov{S}_n +\ov{x} \in K \backslash K_{n, \alpha}\right) \le \frac{C \ln^2n}{n^{p/2 + 1}} u\left(\ov{x}_0 + M \ov{x}\right),
\end{equation}
where
\begin{equation}
\label{Kdef}
K_{n, \alpha}:= \left\{\ov{z} \in K: \mathrm{dist}\left(\ov{z}, \partial K\right) > \alpha \ln n\right\}.
\end{equation}
\end{prop}
\begin{proof}
Let us denote $A_1 =\{\ov{z} \in K: z_1 < \alpha \ln n\}$, $A_2 =\{\ov{z} \in K: z_2 < \alpha \ln n\}$. We fix $\ov{x} \in K$. It is obvious that
$$\pr\left(\tau(\ov{x})> n, \ov{S}_n +\ov{x} \in K \backslash K_{n, \alpha}\right) \le \pr\left(\tau(\ov{x})> n, \ov{S}_n +\ov{x} \in A_1\right) + \pr\left(\tau(\ov{x})> n, \ov{S}_n +\ov{x} \in A_2\right).$$
Let us denote $V:= (-\infty, \alpha \ln n) \times \mathbb{R}$. Note that
\begin{align}
&\pr\left(\tau(\ov{x})> n, \ov{S}_n +\ov{x} \in A_1\right) = \nonumber
\\& = \int\limits_{K} \int\limits_{V} \pr\left(\tau(\ov{x})> n, \ov{S}_n +\ov{x} \in A_1, \ov{S}_{n/3} + \ov{x} \in d\ov{u}, \ov{S}_{n} - \ov{S}_{2n/3} \in d\ov{v}\right) \le \nonumber
\\&\le \int\limits_{K} \int\limits_{-\infty}^{\alpha \ln n} \pr\left(\tau(\ov{x})> n/3, \ov{S}_{n/3} + \ov{x} \in d\ov{u}\right) \pr \left(S_{n/3}^{(1)} \in \left(-v_1 - u_1, -v_1 - u_1 + \alpha \ln n\right)\right) \times \nonumber
\\&\;\;\;\;\;\;\;\;\;\;\;\;\;\;\;\;\;\;\;\;\;\;\;\;\;\;\;\;\;\;\;\;\;\;\;\;\;\;\;\;\;\;\;\; \times \pr \left(S_j^{(1)} > v_1 - \alpha \ln n, j \le n/3, S_{n/3}^{(1)} \in dv_1\right) := I. \label{Idef}
\end{align}
By inequality (\ref{conc}) we have
$$\pr \left(S_{n/3}^{(1)} \in \left(-v_1 - u_1, -v_1 - u_1 + \alpha \ln n\right)\right) \le \frac{C_1 \alpha \ln n}{\sqrt{n}},$$
where $C_1$ does not depend on $u_1$ and $v_1$. Hence, it follows that
\begin{equation}
\label{IEst}
I \le \frac{C_1 \alpha \ln n}{\sqrt{n}} \pr \left( \tau(\ov{x})> n/3 \right) \pr \left(S_j^{(1)} > S_{n/3}^{(1)} - \alpha \ln n, j \le n/3\right).
\end{equation}
Let us denote
$$\widetilde{S}_{j} = S_{n/3}^{(1)} - S_{n/3 - j}^{(1)}, \quad j = 0, 1, \ldots, n/3.$$
By Proposition \ref{lnn}, we have the estimation
\begin{equation}
\label{Stilde}
\pr \left(S_j^{(1)} > S_{n/3}^{(1)} - \alpha \ln n, j \le n/3\right) = \pr \left( \widetilde{S}_j < \alpha \ln n, j \le n/3 \right) \le \frac{C_2 \alpha \ln n}{\sqrt n}.
\end{equation}
Taking into account inequality (\ref{u}), combining relations (\ref{Idef}), (\ref{IEst}), (\ref{Stilde}), we get
$$\pr\left(\tau(\ov{x})> n, \ov{S}_n +\ov{x} \in A_1\right) \le \frac{C \ln^2n}{n^{p/2 + 1}} u\left(\ov{x}_0 + M \ov{x}\right).$$
By the same reasoning for the set $A_2$, we obtain the desired result.
\end{proof}
\begin{lemma}
\label{dimGr}
Let condition (C) be valid for a two-dimensional random walk $\left\{ \ov{S}_n\right\}$, the correlation coefficient of components of the random walk increments $\rho \in (0, 1)$. Let also $D \subset K$ be a measurable bounded set such that $\mathrm{dist}\left(D, \partial K\right) > 0$. Then for any $\alpha > 0$ and any $\ov{x} \in \mathcal{X}_V$ the equality
$$\lim_{T \to \infty} \pr_{\bfx}^{+}\left(\ov{S}_T + \ov{x} \in \sqrt{T}D, \exists j > T: \ov{S}_j + \ov{x} \in K \backslash K_{j, \alpha} \right) = 0$$
holds, where $K_{j, \alpha}$ for $j \in \mathbb{N}$ is defined by (\ref{Kdef}).
\end{lemma}
\begin{proof}
We fix $\ov{x} \in \mathcal{X}_V$. For any $s \in (0, 1/2)$ we have
\begin{align*}
&\pr_{\bfx}^{+}\left(\ov{S}_T + \ov{x} \in \sqrt{T}D, \exists j > T: \ov{S}_j + \ov{x} \in K \backslash K_{j, \alpha} \right) \le
\\ &\le \pr_{\bfx}^{+}\left(\ov{S}_T + \ov{x} \in \sqrt{T}D, \exists j > T: \ov{S}_j + \ov{x} \in K \backslash K_{j, \alpha} \cap U_{T, s}\right) + \pr_{\bfx}^{+}\left(\ov{S}_T + \ov{x} \in \sqrt{T}D, \exists j > T: \ov{S}_j + \ov{x} \notin U_{T, s}\right),
\end{align*}
where $U_{T, s}$ is defined by (\ref{utsDef}). Hence, using Lemma \ref{lemmaU}, we get
\begin{align*}
&\varlimsup_{T \to \infty}\pr_{\bfx}^{+}\left(\ov{S}_T + \ov{x} \in \sqrt{T}D, \exists j > T: \ov{S}_j + \ov{x} \in K \backslash K_{j, \alpha} \right) \le
\\ &\le \varlimsup_{T \to \infty}\pr_{\bfx}^{+}\left(\ov{S}_T + \ov{x} \in \sqrt{T}D, \exists j > T: \ov{S}_j + \ov{x} \in K \backslash K_{j, \alpha} \cap U_{T, s}\right).
\end{align*}
By Theorem \ref{th_plus} the expression
\begin{align*} &\pr_{\bfx}^{+}\left(\ov{S}_T + \ov{x} \in \sqrt{T}D, \exists j > T: \ov{S}_j + \ov{x} \in K \backslash K_{j, \alpha} \cap U_{T, s} \right) =
\\ &= \lim_{n \to \infty} \left. \pr \left( \ov{S}_T + \ov{x} \in \sqrt{T}D,\, \exists j \in (T, n]: \ov{S}_j + \ov{x} \in K \backslash K_{j, \alpha} \cap U_{T, s}\right| \tau(\ov{x}) > n\right)
\end{align*}
takes place. Thus, if we establish the relation
$$\lim_{T \to \infty} \lim_{n \to \infty} n^{p/2} \pr \left( \tau(\ov{x}) > n, \, \ov{S}_T + \ov{x} \in \sqrt{T}D,\, \exists j \in (T, n]: \ov{S}_j + \ov{x} \in K \backslash K_{j, \alpha}\cap U_{T, s}\right) = 0,$$
then Lemma \ref{dimGr} will be proved. We fix natural numbers $n$ and $T$, $n > T$.
It is obvious that
\begin{align}
&\pr \left( \tau(\ov{x}) > n, \, \ov{S}_T + \ov{x} \in \sqrt{T}D,\, \exists j \in (T, n]: \ov{S}_j + \ov{x} \in K \backslash K_{j, \alpha} \cap U_{T, s}\right) \le \nonumber
\\&\le \sum_{j = T} ^{n - 1} \pr \left( \tau(\ov{x}) > n, \, \ov{S}_T + \ov{x} \in \sqrt{T}D,\, \ov{S}_j + \ov{x} \in K \backslash K_{j, \alpha} \cap U_{T, s}\right) + \pr\left(\tau(\ov{x})> n, \ov{S}_n +\ov{x} \in K \backslash K_{n, \alpha}\right). \label{lemma4union}
\end{align}
We fix natural $j \in [T, n-1]$. By the properties of conditional expectation, the equality
\begin{align}
&\pr \left( \tau(\ov{x}) > n, \, \ov{S}_T + \ov{x} \in \sqrt{T}D,\, \ov{S}_j + \ov{x} \in K \backslash K_{j, \alpha} \cap U_{T, s}\right) =\nonumber
\\&= \ex m_{n - j} \left(\ov{S}_j + \ov{x}\right) I\left\{ \tau(\ov{x}) > j, \, \ov{S}_T + \ov{x} \in \sqrt{T}D,\, \ov{S}_j + \ov{x} \in K \backslash K_{j, \alpha} \cap U_{T, s} \right\} \label{lemma4KnDef}
\end{align}
holds. Combining inequalities (\ref{ucUP}) and (\ref{u}), we get the estimation
\begin{equation}
m_{n - j} \left(\ov{S}_j + \ov{x}\right) \le \frac{C_0}{(n - j)^{p/2}} \left|\ov{x}_0 + M\left(\ov{S}_j + \ov{x}\right)\right|^{p-1} \mathrm{dist}\left(\ov{x}_0 + M\left(\ov{S}_j + \ov{x}\right), \partial MK\right). \label{lemma4mn}
\end{equation}
Note that there exist positive parameters $C_1$ and $C_2$ such that for all sufficiently large natural $T$ on the event
$\left\{\ov{S}_j + \ov{x} \in K \backslash K_{j, \alpha} \cap U_{T, s}\right\}$ for $j \ge T$ inequalities
\begin{align*}
&\left|\ov{x}_0 + M\left(\ov{S}_j + \ov{x}\right)\right|^{p-1} \le C_1 \left|\ov{S}_j +\ov{x}\right|^{p-1},
\\& \mathrm{dist}\left(\ov{x}_0 + M\left(\ov{S}_j + \ov{x}\right), \partial MK\right) \le C_2 \mathrm{dist}\left(\ov{S}_j + \ov{x}, \partial K\right) \le \alpha C_2 \ln j
\end{align*}
are valid. Hence, from (\ref{lemma4KnDef}), taking into account (\ref{lemma4mn}), we establish the following relation
\begin{align}
&\pr \left( \tau(\ov{x}) > n, \, \ov{S}_T + \ov{x} \in \sqrt{T}D,\, \ov{S}_j + \ov{x} \in K \backslash K_{j, \alpha} \cap U_{T, s}\right) \le \nonumber
\\&\le \frac{C \ln j}{(n - j)^{p/2}} \ex \left|\ov{S}_j + \ov{x}\right|^{p-1} I\left\{ \tau(\ov{x}) > j, \, \ov{S}_T + \ov{x} \in \sqrt{T}D,\, \ov{S}_j + \ov{x} \in K \backslash K_{j, \alpha} \cap U_{T, s} \right\} \le \nonumber
\\&\le \frac{C \ln j}{(n - j)^{p/2}} \ex \left|\ov{S}_j + \ov{x}\right|^{p-1} I\left\{ \tau(\ov{x}) > j, \, \ov{S}_j + \ov{x} \in K \backslash K_{j, \alpha} \cap U_{T, s} \right\} =: \frac{C \ln j}{(n - j)^{p/2}} E_j. \label{Ejdef}
\end{align}
\ab We fix real positive number $h$ such that
\begin{equation}
\label{ldef}
\frac{2 - p/2}{3 - p} < h < \frac{p/2}{p - 1}.
\end{equation}
It is possible because of $p$ belongs to the interval $(1, 2)$. So, we have
\begin{align*}
E_j &= \ex \left|\ov{S}_j + \ov{x}\right|^{p-1} I\left\{ \tau(\ov{x}) > j, \, \ov{S}_j + \ov{x} \in K \backslash K_{j, \alpha} \cap U_{T, s}, \left|\ov{S}_j + \ov{x}\right| < j^{h}\right\} +
\\ &+\ex \left|\ov{S}_j + \ov{x}\right|^{p-1} I\left\{ \tau(\ov{x}) > j, \, \ov{S}_j + \ov{x} \in K \backslash K_{j, \alpha} \cap U_{T, s}, \left|\ov{S}_j + \ov{x}\right| \ge j^{h}\right\} =: E_{j, 1} + E_{j, 2}.
\end{align*}
By inequality (\ref{Kn_uneq}), the inequality
\begin{align*}
E_{j, 1} &\le j^{h(p - 1)} \pr\left(\tau(\ov{x})> j, \ov{S}_j +\ov{x} \in K \backslash K_{j, \alpha}\right) \le C_3 (\ln j)^{2} \cdot j ^{h(p-1) - p/2 - 1} u\left(\ov{x}_0 + M\ov{x}\right)
\end{align*}
takes place. From the second inequality in (\ref{ldef}), it follows that $h(p-1)-p/2 -1 < -1$. Thus, there exists a positive number $\delta_1$ such that for all sufficiently large $j$ the following estimation holds:
\begin{equation}
\ln j \cdot E_{j, 1}\le \frac{C_4 u\left(\ov{x}_0 + M\ov{x}\right)}{j^{1 + \delta_1}}. \label{Ej1}
\end{equation}
Also the inequalities
\begin{align}
E_{j, 2} &\le j^{h(p - 3)} \ex \left|\ov{S}_j + \ov{x} \right|^{2} I\left\{ \tau(\ov{x}) > j, \, \ov{S}_j + \ov{x} \in K \backslash K_{j, \alpha} \cap U_{T, s}, \left|\ov{S}_j + \ov{x}\right| \ge j^{h}\right\} \nonumber
\\&\le j^{h(p - 3)} \ex \left|\ov{S}_j + \ov{x} \right|^{2} I\left\{ \tau(\ov{x}) > j\right\} \le j^{h(p - 3)} \ex \left|\ov{x} + \ov{S}_{\tau(\ov{x}) \wedge j}\right|^2 \label{E21}
\end{align}
take place. Given $\ov{x} \in K$, from (\ref{eq_for_exp}) and (\ref{stop}), for all $j$ we have
\begin{align*}
\ex \left|\ov{x} + \ov{S}_{\tau(\ov{x}) \wedge j}\right|^2 &= \left|\ov{x}\right|^2 + \ex \left|\ov{X}\right|^2 \ex \left[\tau(\ov{x})\wedge j\right] \le \left|\ov{x}\right|^2 + C_5 j^{1 - p/2} u\left(\ov{x}_0 + M \ov{x}\right) \le \nonumber
\\ &\le C_6 j^{1 -p/2} u\left(\ov{x}_0 + M\ov{x}\right). 
\end{align*}
Hence, from inequality (\ref{E21}), for all sufficiently large $j$
$$E_{j, 2} \le C_7 j^{h(p-3) + 1 - p/2} u\left(\ov{x}_0 + M\ov{x}\right).$$
Taking into account the inequality $h(p-3) + 1 - p/2 < -1$, we get that there exists a positive number $\delta_2$ such that
\begin{equation}
\ln j \cdot E_{j, 2}\le \frac{C_8u\left(\ov{x}_0 + M \ov{x}\right) }{j^{1 + \delta_2}} \label{Ej2}
\end{equation}
for all sufficiently large $j$. Let us denote $\delta: = \min\{\delta_1, \delta_2\} > 0$. From inequalities (\ref{Ej1}) and (\ref{Ej2}), for all sufficiently large $j$ we have
\begin{equation}
\ln j \cdot E_j \le \frac{C(\ov{x})}{j^{1 + \delta}}. \label{Ej}
\end{equation}
If $T$ has been chosen sufficiently large, then for all natural $j \in [T, n-1]$ from relations (\ref{Ejdef}) and (\ref{Ej}), we have the inequality
\begin{equation*}
\pr \left( \tau(\ov{x}) > n, \, \ov{S}_T + \ov{x} \in \sqrt{T}D,\, \ov{S}_j + \ov{x} \in K \backslash K_{j, \alpha} \cap U_{T, s}\right) \le \frac{\widetilde{C}(\ov{x})}{j^{1 + \delta} (n - j)^{p/2}}, \label{lemma4main}
\end{equation*}
where $\widetilde{C}(\ov{x})$ does not depend on both $n$ and $j$. Hence,
\begin{align*}
&\varlimsup_{T \to \infty} \lim_{n \to \infty} n^{p/2} \pr \left( \tau(\ov{x}) > n, \, \ov{S}_T + \ov{x} \in \sqrt{T}D,\, \exists j \in (T, n]: \ov{S}_j + \ov{x} \in K \backslash K_{j, \alpha}\right) \le
\\&\le \varlimsup_{T \to \infty} \varlimsup_{n \to \infty} n^{p/2} \widetilde{C}(\ov{x}) \sum_{j = T}^{n - 1} \frac{1}{j^{1 + \delta}(n - j)^{p/2}} + \varlimsup_{T \to \infty} \varlimsup_{n \to \infty} n^{p/2} \pr\left(\tau(\ov{x})> n, \ov{S}_n +\ov{x} \in K \backslash K_{n, \alpha}\right).
\end{align*}
By inequality (\ref{Kn_uneq}), we have
$$\lim_{n \to \infty} n^{p/2} \pr\left(\tau(\ov{x})> n, \ov{S}_n +\ov{x} \in K \backslash K_{n, \alpha}\right) = 0.$$
Note that
\begin{align}
&n^{p/2} \sum_{j = T}^{n/2 - 1} \frac{1}{j ^{1 + \delta}(n - j)^{p/2}} \le 2^{p/2} \sum_{j = T}^{n/2 -1} \frac{1}{j ^{1 + \delta}} \le 2^{p/2} \sum_{j = T}^{+\infty} \frac{1}{j ^{1 + \delta}}; \label{lemma4first}
\\&n^{p/2} \sum_{j = n/2}^{n-1} \frac{1}{j ^{1 + \delta}(n - j)^{p/2}} \le 2^{1 + \delta} n^{p/2 - 1 - \delta}
\sum_{j = 1}^{n/2} \frac{1}{j^{p/2}} \le C n^{p/2 - 1 - \delta} n^{1 - p/2} \le \frac{C}{n^{\delta}}. \label{lemma4second}
\end{align}
Let us remark that the right-hand side of inequality (\ref{lemma4first}) tends to zero as $T \to \infty$, at the same time the right-hand side of inequality (\ref{lemma4second}) tends to zero as $n \to \infty$. Thus, Lemma \ref{dimGr} is proved.
\end{proof}
\begin{cons}
\label{granitsa_pr_plus}
Let condition (C) be valid for a two-dimensional random walk $\left\{ \ov{S}_n\right\}$, the correlation coefficient of components of the random walk increments $\rho \in (0, 1)$. Then for any $\ov{x} \in \mathcal{X}_V$ and any $\alpha > 0$ the equality
$$\lim_{T \to \infty} \pr_{\bfx}^{+}\left(\exists j > T: \ov{S}_j + \ov{x} \in K \backslash K_{j, \alpha}\right) = 0$$
holds, where
$$K_{j, \alpha} = \left\{\ov{x} \in K: \mathrm{dist}\left(\ov{x}, \partial K\right) > \alpha \ln j\right\}.$$
\end{cons}
\begin{proof}
We fix $\varepsilon > 0$ and a vector $\ov{x} \in \mathcal{X}_V$. We find a square $D$ such that inequality (\ref{not_in_sq_p_plus}) holds. By Lemma \ref{dimGr}, we have
\begin{align*}
\varlimsup_{T \to \infty} \pr_{\bfx}^{+}\left(\exists j > T: \ov{S}_j + \ov{x} \in K \backslash K_{j, \alpha}\right) &\le \varlimsup_{T \to \infty} \pr_{\bfx}^{+}\left( \ov{S}_T + \ov{x} \notin \sqrt{T}D\right) \le \varepsilon \frac{u\left(\ov{x}_0 + M\ov{x}\right)}{V\left(M\ov{x}\right)}.
\end{align*}
Tending $\varepsilon \downarrow 0$, we have
$$\lim_{T \to \infty} \pr_{\bfx}^{+}\left(\exists j > T: \ov{S}_j + \ov{x} \in K \backslash K_{j, \alpha}\right) = 0.$$
Corollary \ref{granitsa_pr_plus} is proved.
\end{proof}
\begin{cons}
\label{eSn_to_zero_plus}
Let condition (C) be valid for a two-dimensional random walk $\left\{ \ov{S}_n\right\}$, the correlation coefficient of components of the random walk increments $\rho \in (0, 1)$. Then for any $\ov{x} \in \mathcal{X}_V$ the relations
$$\lim_{n \to \infty} S_n^{(i)} = +\infty, \quad \pr_{\bfx}^{+}-\text{a.s.}, i = 1, 2$$
hold.
\end{cons}
\subsection{Proof of Lemma \ref{lemma7}}
\begin{prop}
Let $\left\{\ov{Z}_n\right\}$ be a group of $N$ branching processes in a joint random environment. Then,
\begin{itemize}
\item the following expressions
\begin{equation}
\label{prf_expression}
\prf\left(Z_n^{(i)} > 0 \right) = \left(\exp\left(-S_n^{(i)}\right) + \sum_{k = 0}^{n-1}\exp\left(-S_k^{(i)}\right) g_{k}^{(i)}\left( f_{k+2, n}^{(i)}(0)\right)\right)^{-1}, \quad i = 1, \ldots, N,
\end{equation}
are valid, where
\begin{align*}
&f_{k, n}^{(i)} (t): = f_k^{(i)} \circ f_{k+1}^{(i)} \circ \ldots \circ f_n^{(i)}(t), \quad k = 1, 2, \ldots, n, \quad f_{n+1, n}^{(i)}(t) := t, \quad i = 1, \ldots N,
\\&g_k^{(i)}(t) := \frac{1}{1 - f_{k+1}^{(i)}(t)} - \frac{1}{\left(f_{k+1}^{(i)}\right)'(1) \cdot (1 - t)}, \quad t \in [0, 1), \; i = 1, \ldots, N;
\end{align*}
\item the inequalities
\begin{equation}
\label{prf_un_nogeom}
\prf \left(Z_n^{(i)} > 0\right) \ge \left(\mathrm{exp} \left(-S_n^{(i)}\right) + \sum_{k = 0}^{n-1} \xi_k^{(i)} \mathrm{exp} \left(-S_k^{(i)}\right)\right)^{-1}, \quad i = 1, \ldots, N
\end{equation}
hold.
\end{itemize}
\end{prop}
\begin{proof}
See Lemma 19.5 in \cite{Afan}.
\end{proof}
In what follows, we agree to write $q$ for the parameter from Theorem \ref{main_nogeom}. In other words,
$$\ex \left(\xi_0^{(1)}\right)^q < \infty, \quad \ex \left(\xi_0^{(2)}\right)^q < \infty.$$
We need the following technical lemma to prove Lemma \ref{lemma7}.
\begin{lemma}
\label{bor_xi}
Under the conditions of Theorem \ref{main_nogeom} for any $\beta > 2/(q (2 - p))$ and any $\ov{x} \in \mathcal{X}_V$
$$\sum_{k = 1}^{\infty} \pr_{\bfx}^{+}\left( \xi_k^{(i)} > k^{\beta}\right) < \infty, \quad i = 1, 2.$$
\end{lemma}
\begin{proof}
We fix a vector $\ov{x} \in \mathcal{X}_V$, a natural number $k$, $i \in \left\{1, 2\right\}$. To simplify the notation we omit the superscripts of $\xi_k^{(i)}$. By the definition of $\pr^{+}$-measure, we have
\begin{equation}
\label{series_px_exp}
V\left(M\ov{x}\right) \pr_{\bfx}^{+}\left(\xi_k > k^{\beta}\right) = \ex V\left(M \left(\ov{x} + \ov{S}_{k+1}\right)\right)I\{\tau(\ov{x}) > k+1\}I\{\xi_{k} > k^{\beta}\}.
\end{equation}
Let us denote $\ov{w}^{+} = \left(w_1 I\{w_1 \ge 0\}, w_2 I\{w_2 \ge 0 \}\right)$ for any vector $\ov{w} = (w_1, w_2)$.
From relations (\ref{V_mon}), (\ref{VU_un_eq}), it follows that
\begin{align}
V\left(M\left(\ov{x} + \ov{S}_{k+1}\right)\right) &\le V\left(M(\ov{x} + \ov{S}_k) + \left(M\ov{X}_{k+1}\right)^{+}\right) \le Cu\left(M\left(\ov{x} + \ov{S}_k\right) + \ov{x}_0 + \left(M\ov{X}_{k+1}\right)^{+} \right). \label{Vu_plus}
\end{align}
Note that for any $\ov{z}_1, \ov{z}_2 \in MK$ the inequalities $\dist(\ov{z}_1 + \ov{z}_2, \partial MK) \ge \dist(\ov{z}_i, \partial MK)$, $i = 1, 2$ hold. Thus, if $\ov{x} + \ov{S}_{k} \in K$, then
$$\dist\left(M\left(\ov{x} + \ov{S}_k\right) + \ov{x}_0, \partial MK \right) \ge \dist(\ov{x}_0, \partial MK) > 0.$$
Using inequality (\ref{uNPart}), we get
\begin{align*}
u\left(M\left(\ov{x} + \ov{S}_k\right) + \ov{x}_0 + \left(M\ov{X}_{k+1}\right)^{+} \right) &\le C_1 \max\left(\left|\left(M\ov{X}_{k+1}\right)^{+}\right|^p, 1\right) u\left(M\left(\ov{x} +\ov{S}_{k}\right) + \ov{x}_0 \right).
\end{align*}
Hence, from equality (\ref{series_px_exp}) and inequality (\ref{Vu_plus}), we obtain the following relations:
\begin{align}
&V\left(M\ov{x}\right) \pr_{\bfx}^{+}\left(\xi_k > k^{\beta}\right) \le \nonumber
\\ \le &CC_1 \ex u\left(M\left(\ov{x} +\ov{S}_{k}\right) + \ov{x}_0 \right)I\{\tau(\ov{x}) > k+1\} \max\left(\left|\left(M\ov{X}_{k+1}\right)^{+}\right|^p, 1\right) I\{\xi_k > k^{\beta}\} \le \nonumber
\\ \le& CC_1 \ex u\left(M\left(\ov{x} +\ov{S}_{k}\right) + \ov{x}_0 \right)I\{\tau(\ov{x}) > k\} \cdot \ex \max\left(\left|\left(M\ov{X}_{k+1}\right)^{+}\right|^p, 1\right) I\{\xi_k > k^{\beta}\}. \label{pr_plus_xi_ind}
\end{align}
Moreover,
\begin{align}
&\varlimsup_{k \to \infty} \ex u\left(M\left(\ov{x} +\ov{S}_{k}\right) + \ov{x}_0 \right)I\left\{\tau(\ov{x}) > k\right\} \le \nonumber
\\ &\le \lim_{k \to \infty}\ex u\left(M\left(\ov{x} +\ov{S}_{k}\right) + \ov{x}_0 \right)I\left\{\tau\left(\ov{x} + M^{-1}\ov{x}_0\right) > k\right\} = V(M\ov{x} + \ov{x}_0), \nonumber
\end{align}
where the last equality follows from (\ref{V2}). Hence, we get that there exists a number $C_2$ such that for any natural $k$ the estimation
\begin{equation}
\label{u_high}
\ex u\left(M\left(\ov{x} +\ov{S}_{k}\right) + \ov{x}_0 \right)I\left\{\tau(\ov{x}) > k\right\} \le C_2 V(M\ov{x} + \ov{x}_0)
\end{equation}
takes place. For the second multiplier in the right-hand side of relation (\ref{pr_plus_xi_ind}) we have
$$\ex \max\left(\left|\left(M\ov{X}_{k+1}\right)^{+}\right|^p, 1\right) I\{\xi_k > k^{\beta}\} \le \pr(\xi_k > k^{\beta}) + \ex \left|\left(M\ov{X}_{k+1}\right)^{+}\right|^p I\{\xi_k > k^{\beta}\}.$$
Using Holder's inequality, we get
$$\ex \left|\left(M\ov{X}_{k+1}\right)^{+}\right|^p I\{\xi_k > k^{\beta}\} \le \left(\ex \left|\left(M\ov{X}_{k+1}\right)^{+}\right|^2\right)^{p/2} \cdot \left(\pr\left(\xi_k > k^{\beta}\right)\right)^{(2-p)/2}.$$
It is obvious that
$$\ex \left|\left(M\ov{X}_{k+1}\right)^{+}\right|^2 \le \ex \left|M\ov{X}_{k+1}\right|^2 < \infty.$$
By Markov's inequality, the estimation
$$\pr\left(\xi_k > k^{\beta}\right) \le \frac{\ex \xi^q}{k^{\beta q}}$$
holds. Thus,
\begin{align}
\label{max_exp_xi}
\ex \max\left(\left|\left(M\ov{X}_{k+1}\right)^{+}\right|^p, 1\right) I\{\xi_k > k^{\beta}\} \le \frac{\ex \xi^q}{k^{\beta q}} + C_3 \left( \frac{1}{k^{\beta q}}\right)^{(2-p)/2} \le \frac{C_4}{k^{1 + \varepsilon}}
\end{align}
for some $\varepsilon > 0$. By (\ref{pr_plus_xi_ind}), taking into account relations (\ref{u_high}), (\ref{max_exp_xi}), we may conclude that for any natural $k$ the following estimation is valid:
$$\pr_{\bfx}^{+}\left(\xi_k > k^{\beta}\right) \le CC_1C_2C_4 \frac{V\left(M\ov{x} + \ov{x}_0\right)}{V\left(M\ov{x}\right)} \frac{1}{k^{1 + \varepsilon}} = \frac{C(\ov{x})}{k^{1 + \varepsilon}}.$$
Summation the last relation over $k$ leads us to the statement of Lemma \ref{bor_xi}.
\end{proof}
Finally, we are ready to prove Lemma \ref{lemma7}.
\begin{proof}[Proof of Lemma \ref{lemma7}]
We fix $\ov{x} \in \mathcal{X}_V$. By expression (\ref{prf_expression}), we may conclude that $\prf\left(\ov{Z}_n > 0\right)$ is the $\mathcal{F}_n$-measurable random variable for every natural $n$. It is obvious that the sequence $\left\{\prf\left(\ov{Z}_n > 0 \right)\right\}$ converges $\pr_{\bfx}^{+}$-a.s.. By Theorem \ref{th_plus}, we have
$$r(\ov{x}) := \lim_{n \to \infty} \left.\pr \left( \ov{Z}_n > 0 \right.|\tau(\ov{x}) > n \right) = \lim_{n \to \infty} \ex \left.\left(\prf\left(\ov{Z}_n > 0\right)\right|\tau(\ov{x})> n\right) = \ex_{\bfx}^{+}\lim_{n \to \infty} \prf\left(\ov{Z}_n > 0\right).$$
Our next purpose is to prove that $r(\ov{x}) > 0$. We fix $\beta= 1 + 2 / (q(2 - p))$, $\alpha = 2 + \beta$. Let us denote
\begin{align*}
&A_T := \left\{\xi_k^{(i)} \le u, \exp\left(-S_k^{(i)}\right) \le w, i = 1, 2, k = 0, \ldots, T - 1\right\},
\\&B_T := \left\{ \xi_k^{(i)} \le k^{\beta}, k \ge T, i = 1, 2\right\}, \quad G_T := \left\{\ov{S}_j + \ov{x} \in K_{j, \alpha}, j \ge T \right\},
\\&K_{j, \alpha} := \left\{\ov{z} \in K: \mathrm{dist}\left(\ov{z}, \partial K\right) > \alpha \ln j\right\}
\end{align*}
for any integer non-negative $T$ and any real $u$ and $w$. Note that the following relations
\begin{align}
&\pr_{\bfx}^{+}\left(A_T \cap B_T \cap G_T\right) = \nonumber
\\&= \pr_{\bfx}^{+}\left(A_T\right) - \pr_{\bfx}^{+}\left(A_T \cap \ov{B}_T\right) - \pr_{\bfx}^{+}\left(A_T\cap B_T \cap \ov{G}_T\right) \ge \pr_{\bfx}^{+}(A_T) - \pr_{\bfx}^{+}\left(\ov{B}_T\right) - \pr_{\bfx}^{+}\left(\ov{G}_T\right) \label{down_pr_plus}
\end{align}
take place. We fix $\varepsilon \in (0, 1/3)$. By corollary \ref{granitsa_pr_plus}, we may find a number $T_1$ such that for any $T > T_1$ :
\begin{equation}
\label{GT_plus}
\pr_{\bfx}^{+}\left(\ov{G}_T\right) = \pr_{\bfx}^{+}\left(\exists j \ge T: \ov{S}_j + \ov{t} \in K \backslash K_{j, \alpha}\right)\le \varepsilon.
\end{equation}
By Lemma \ref{bor_xi}
$$\sum_{k} \pr_{\bfx}^{+}\left(\xi_{k}^{(i)} > k^{\beta}\right)< \infty, \quad i = 1, 2.$$
Hence, there exists a number $T_2$ such that for any $T > T_2$ the inequality
\begin{equation}
\label{BT_plus}
\pr_{\bfx}^{+}\left( \ov{B}_T\right) \le \sum_{k = T}^{\infty} \sum_{i = 1}^2 \pr_{\bfx}^{+}\left(\xi_k^{(i)} > k^{\beta}\right) < \varepsilon
\end{equation}
holds. We fix $T > \max \left(T_1, T_2\right)$. We choose $u$ and $w$ such that
\begin{equation}
\label{AT_plus}
\pr_{\bfx}^{+}(A_T) \ge 1 - \varepsilon.
\end{equation}
Finally, from relation (\ref{down_pr_plus}), taking into account inequalities (\ref{GT_plus}), (\ref{BT_plus}), (\ref{AT_plus}), we get the estimation
\begin{equation*}
\label{AT_BT_FT_GT_plus}
\pr_{\bfx}^{+}\left(A_T \cap B_T \cap G_T\right) \ge 1 - 3\varepsilon > 0.
\end{equation*}
\ab By inequality (\ref{prf_un_nogeom}) and corollary \ref{eSn_to_zero_plus}, we have
$$r(\ov{x}) = \ex_{\bfx}^{+} \lim_{n \to \infty} \prf\left(\ov{Z}_n > 0\right) \ge \ex_{\bfx}^{+} \zeta^{(1)}\zeta^{(2)},\quad \zeta^{(i)} := \left(\sum_{k = 0}^{\infty} \xi_k^{(i)} \exp\left( -S_k^{(i)}\right) \right)^{-1}, \quad i = 1, 2.$$
Note that for $i \in \left\{ 1, 2 \right\}$ the following chain of equalities
\begin{align}
\zeta^{(i)}I\left\{A_T \cap B_T \cap G_T\right\} &\ge \left(ulT + \sum_{k = T}^{\infty} \frac{e^{x_i}}{k^{\alpha - \beta}}\right)^{-1}I\left\{A_T \cap B_T \cap G_T\right\} \ge \left(ulT + 2 e^{x_i}\right)^{-1}I\left\{A_T \cap B_T \cap G_T\right\} \nonumber
\end{align}
takes place. Hence,
\begin{align*}
r(\ov{x}) \ge \ex_{\bfx}^{+} \zeta^{(1)}\zeta^{(2)} \ge \frac{1}{\left(ulT + 2 e^{x_1}\right) \left(ulT + 2 e^{x_2}\right)} \pr_{\bfx}^{+}\left(A_T \cap B_T \cap G_T \right) > 0,
\end{align*}
which proves Lemma \ref{lemma7}.
\end{proof}
\section{Acknowledgments}
I take the opportunity to express my profound gratitude to A.V. Shklyaev for discussions and advices on this work and to  Professor V.A. Vatutin for remarks.
\bibliography{nonExtinEng}
\end{document}